\documentclass[12pt]{amsart}

\usepackage{amssymb, amscd, txfonts}
\usepackage[all]{xy}
\usepackage[dvips]{graphicx, xcolor}
\usepackage{tikz}
\usetikzlibrary{calc,intersections,patterns}
\usepackage{enumitem}
\usepackage[a4paper, hmarginratio={1:1}, vmarginratio={1:1}, textwidth=16cm, textheight=22.5cm]{geometry}

\numberwithin{equation}{section}

\newtheorem{thm}{Theorem}[section]
\newtheorem{prop}[thm]{Proposition}
\newtheorem{lem}[thm]{Lemma}
\newtheorem{cor}[thm]{Corollary}
\newtheorem{q}[thm]{Question}
\newtheorem*{thm*}{Main Theorem}

\theoremstyle{definition}
\newtheorem{defn}[thm]{Definition}

\theoremstyle{remark}
\newtheorem{rem}[thm]{Remark}
\newtheorem{exmp}[thm]{Example}

\renewcommand{\hom}{\operatorname{Hom}}

\newcommand{\Z}{\mathbb{Z}}

\newcommand{\R}{\mathbb{R}}
\newcommand{\C}{\mathbb{C}}

\DeclareMathOperator{\tr}{tr}

\begin{document}

\title[Character varieties and splittings of $3$-manifolds]
{Character varieties of higher dimensional representations and splittings of $3$-manifolds}
\author[T.~Hara]{Takashi Hara}
\author[T.~Kitayama]{Takahiro Kitayama}
\address{Department of Mathematics, 
College of Liberal Arts, Tsuda University, 2-1-1 Tsuda-machi, Kodaira-shi, Tokyo 157-8577, Japan}
\email{t-hara@tsuda.ac.jp}
\address{Graduate School of Mathematical Sciences, the University of Tokyo,
3-8-1 Komaba, Meguro-ku, Tokyo 153-8914, Japan}
\email{kitayama@ms.u-tokyo.ac.jp}

\begin{abstract}
In 1983 Culler and Shalen established a way to construct essential surfaces 
in a $3$\nobreakdash-manifold from ideal points of the $\mathrm{SL}_2$-character variety
 associated to the $3$\nobreakdash-manifold group.
We present in this article an analogous construction
of certain kinds of branched surfaces
(which we call {\em essential tribranched surfaces})
from ideal points of the $\mathrm{SL}_n$-character variety 
for a natural number~$n$ greater than or equal to $3$.
Further we verify that
such a branched surface induces a nontrivial presentation
of the $3$\nobreakdash-manifold group
in terms of the fundamental group of 
a certain $2$-dimensional complex of groups.
\end{abstract}

\maketitle

\setcounter{section}{-1}
\section{Introduction} \label{sc:introduction}

In their notable work~\cite{CS} Culler and Shalen established a method
to construct {\em essential surfaces} in a $3$-manifold from information of the $\mathrm{SL}_2(\C)$-character variety of its fundamental group.
The method is based upon the interplay among hyperbolic geometry, the
theory of incompressible surfaces and the theory on the structure 
of subgroups of the special linear group $\mathrm{SL}(2)$ of degree~$2$.
Culler--Shalen theory provides a basic and  powerful tool in
low-dimensional topology, and it has given fundamentals 
for many significant breakthroughs; for example, 
Culler and Shalen themselves proved 
the generalised Smith conjecture as a special case of their main 
results in \cite{CS}. Meanwhile, Morgan and Shalen~\cite{MS1, MS2, MS3} proposed 
new understandings of Thurston's results: the characterisation of
$3$-manifolds with the compact space of hyperbolic structures
\cite{Thurston1} and a compactification of the Teichm\"{u}ller space 
of a surface \cite{Thurston2}. Further Culler, Gordon, Luecke and Shalen~\cite{CGLS} proved the cyclic surgery theorem on Dehn fillings of knots.
We refer the reader to the exposition \cite{Shalen} for more literature and related topics on Culler--Shalen theory.

The aim of this article is to present a theory analogous to Culler and
Shalen's for {\em higher dimensional representations} of the
$3$-manifold group.
We first introduce a special kind of branched surfaces embedded in a
$3$-manifold, which we call an {\em essential tribranched surface} (see
Definition~\ref{def:ETBS} for details), 
and observe that it induces 
a nontrivial presentation of the
$3$\nobreakdash-manifold group in terms of the fundamental group 
of a certain $2$-dimensional {\em complex of groups} (see Section~\ref{sc:CG}
for the definition of complexes of groups).
Then we show that an essential tribranched surface is constructed from
an {\em ideal point} of an affine curve 
in the $\mathrm{SL}_n(\C)$-character variety of the $3$-manifold group.
Note that in our terminology an essential surface (in the usual sense) 
can be regarded as an essential tribranched surface without any branched points.

We here explain our strategy to construct an essential tribranched
surface in more detail.
Let $M$ be a compact, connected, irreducible and  orientable $3$-manifold.
We suppose that the $\mathrm{SL}_n(\C)$-character variety $X_n(M)$ of $\pi_1 (M)$
is of positive dimension, and let $\tilde{x}$ be an ideal point of an
affine algebraic curve $C$ in $X_n(M)$. By construction 
$X_n(M)$ is obtained as the (geometric invariant theoretical) quotient 
of the affine algebraic set $\hom(\pi_1(M),\mathrm{SL}_n(\C))$ by the conjugate action of $\mathrm{SL}_n(\C)$, 
and we may take a lift $D$ of $C$ in $\hom(\pi_1(M),\mathrm{SL}_n(\C))$. 
Let $\tilde{y}$ be a ``lift'' of $\tilde{x}$, which is an ideal point of the affine curve $D$.
We denote by $\C(D)$ the field of rational functions on $D$. 
The construction of an essential tribranched surface from $\tilde{x}$ is
divided into the following three steps.
Firstly, on the basis of the theory of {\em Bruhat--Tits buildings}
elaborated by Iwahori and Matsumoto \cite{IM}, and Bruhat and Tits
\cite{BT1,BT2}, we may associate to the ideal point $\tilde{y}$ 
a canonical action of $\mathrm{SL}_n(\C(D))$ on
an $(n-1)$-dimensional Euclidean building
$\mathcal{B}_{n,\widetilde{D}, \tilde{y}}$  
(see Section~\ref{ssc:NAB} for details).
Pulling back this canonical action by the tautological representation
$\pi_1(M) \to \mathrm{SL}_n(\C(D))$, we obtain an action of $\pi_1(M)$ on
$\mathcal{B}_{n, \widetilde{D}, \tilde{y}}$.
Secondly, we prove that this action is {\em nontrivial}, that is, the
isotropic subgroup at each vertex of $\mathcal{B}_{n,\widetilde{D}, \tilde{y}}$ with
respect to this action is a proper subgroup of $\pi_1(M)$.
The important point to note here is that in the case of $n = 2$ this
step is an algebraic heart of Culler and Shalen's original work \cite[Theorem~2.2.1]{CS}.
Thirdly, we show that one can construct an essential tribranched surface in general from a nontrivial action of $\pi_1(M)$ on a Euclidean building.
In this step we consider certain modifications of classical techniques due to
Stallings and Waldhausen for constructing an essential surface 
as a {\em dual} of a nontrivial action of $\pi_1(M)$ on a tree.

Now let $\mathcal{B}_{n, \widetilde{D}, \tilde{y}}^{(2)}$ denote the $2$-skeleton of
the Bruhat--Tits building $\mathcal{B}_{n,\widetilde{D},\tilde{y}}$ and let  
$Y(\mathcal{B}_{n, \widetilde{D}, \tilde{y}}^{(2)} / \pi_1(M))$ denote 
the $1$-dimensional subcomplex of the first barycentric subdivision 
of the quotient complex $\mathcal{B}_{n, \widetilde{D}, \tilde{y}}^{(2)} / \pi_1(M)$
consisting of all the barycentres of $1$- and $2$-simplices
and all the edges connecting them.
We say that an ideal point $\tilde{x}$ of an affine curve in $X_n(M)$
{\em gives a tribranched surface} $\Sigma$ if there exists 
a map $f\colon M \to \mathcal{B}_{n, \widetilde{D}, \tilde{y}} / \pi_1(M)$ such that the
tribranched surface $\Sigma$ coincides with the inverse
image of $Y(\mathcal{B}_{n, \widetilde{D}, \tilde{y}}^{(2)} / \pi_1(M))$ under $f$.
The main theorem of this article is as follows:

\begin{thm*}[Theorem \ref{thm:main}]
Let $n$ be a natural number greater than or equal to $3$, 
and assume that the boundary $\partial M$ of $M$ is non-empty when 
$n$ is strictly greater than $3$.
Then an ideal point of an affine algebraic curve in $X_n(M)$ 
gives an essential tribranched surface in $M$. 
\end{thm*}

The assumption on the boundary of $M$ comes 
from a certain technical reason required in the proof of the main result.
See the proof of Theorem~\ref{thm:BT} for details.

This article is organised as follows.
In Section~\ref{sc:CG} we give a brief exposition on complexes of groups.
Section~\ref{sc:TBS} is devoted to introduce the notion of essential tribranched
surfaces and to describe splittings of the $3$\nobreakdash-manifold
groups induced by an essential tribranched surface.
In Section~\ref{sc:BT} we review fundamentals on Bruhat--Tits buildings, in
particular, for the special linear groups.
In Section~\ref{sc:construction} the main theorem stated above is proved.
We first review several standard facts on $\mathrm{SL}_n(\C)$-character varieties in
Section~\ref{ssc:charvar}.
We then show in Section~\ref{ssc:NAB} that the action of the $3$\nobreakdash-manifold group on the Bruhat--Tits
building associated to an ideal point is nontrivial, and construct
an essential tribranched surface from such a nontrivial action in Section~\ref{ssc:construction}.   
Section~\ref{sc:A} provides an application of the theory of this article to small Seifert manifolds.
In Section~\ref{sc:Q} we raise several questions to be further studied. 

The contents of Sections~\ref{sc:CG}, \ref{ssc:cgtbs} and
\ref{ssc:SETBS} (concerning complexes of groups associated to
essential tribranched surfaces) are rather independent 
of other parts of this article,
and hence readers who are only interested in the construction of nontrivial essential tribranched surfaces may skip these sections and proceed to Section~\ref{sc:construction}.

\bigskip
\paragraph{\bfseries Note}
After the first version of this article appeared on the arXiv, Question~\ref{q:detection} in Section~\ref{sc:Q} was solved affirmatively in a much stronger form by Friedl, Nagel and the second-named author~\cite{FKN2}.
In fact, based on the construction of essential tribranched surfaces developed in this article, the breakthroughs of Agol~\cite{A} and Wise~\cite{W} on the separability of subgroups in a $3$-manifold group and the subsequent works of Przytycki and Wise~\cite{PW1, PW2}, they proved that {\em every} connected essential surface (without any branched points) in $M$ is given by an ideal point of a rational curve in $X_n(M)$ for some natural number $n$.

\bigskip
\paragraph{\bfseries Acknowledgments}
The authors would like to express their sincere gratitude to Masanori Morishita
for his valuable comments suggesting
that their results might be extended in the direction of arithmetic topology.
They would also like to thank Steven Boyer for several helpful comments
drawing their attention to the contents of Section \ref{sc:A},
and Stefan Friedl, Matthias Nagel, Tomotada Ohtsuki, Makoto Sakuma and Yuji Terashima for helpful suggestions.
This research was supported by JSPS Research Fellowships (Grant-in-Aids for Early-Carrer Scientists: Grant Numbers 18K13395 and 18K13404, and Fund for the Promotion of Joint International Research (Fostering Joint International Research (A)): Grant Number 18KK0380).

\tableofcontents

\section{Preliminaries on complexes of groups} \label{sc:CG}

The theory of graphs of groups due to Hyman Bass and Jean-Pierre Serre \cite{Serre}
has been naturally generalised to the {\em theory of complexes of groups}
introduced independently by Jon Michael Corson \cite{Corson} (mainly for 
$2$\nobreakdash-complexes of groups) 
and Andr\'e Haefliger \cite{Haefliger} (in general).
We shall briefly recall the definitions of complexes of groups and 
their fundamental groups. Here we adopt a combinatorial 
approach proposed in \cite[Chapter~III.$\mathcal{C}$]{BH} 
rather than a topological approach based upon the concept of 
{\em complexes of spaces} especially when we define the
fundamental groups of complexes of groups (see \cite{Corson, Haefliger} 
for details of the latter approach). One of the great virtues of 
the combinatorial approach is that one may explicitly describe
generators and relations of the fundamental group of a complex of
groups, as we shall see later in Section~\ref{ssc:FG}. 
Furthermore we shall consider 
complexes of groups over {\em scwols} rather than complexes of groups 
over combinatorial CW-complexes (the latter notion is 
introduced in \cite[Section~2, Definition]{Corson}). 
One may readily observe that these two concepts of 
complexes of groups essentially coincide by considering scwols 
associated to combinatorial CW-complexes (essentially it is equivalent
to consider the ``first barycentric subdivision'' 
of a combinatorial CW-complex). 
We shall briefly explain how to associate 
a scwol to a combinatorial CW-complex (of dimension~$2$) via the first barycentric subdivision later in Section~\ref{ssc:cgtbs}.

\subsection{Scwols and their fundamental groups} 

Recall that a {\em scwol} $\mathcal{Y}$ (an abbreviation of a \underline{s}mall
\underline{c}ategory \underline{w}ith\underline{o}ut \underline{l}oops)
consists of two sets $V(\mathcal{Y})$ and $E(\mathcal{Y})$ equipped with 
set-theoretical maps $i \colon E(\mathcal{Y}) \rightarrow
V(\mathcal{Y})$ and $t \colon E(\mathcal{Y}) \rightarrow V(\mathcal{Y})$ 
satisfying the following properties:
\begin{enumerate}[labelindent=.5em, labelwidth=3.5em, labelsep*=1em,
 leftmargin=!, label=(Scw\arabic*)]
\setcounter{enumi}{0}
\item an element of $E(\mathcal{Y})$ denoted by $ab$ is associated to 
      each pair $(a,b)$ of elements of $E(\mathcal{Y})$
      satisfying $i(a)=t(b)$
      (the element $ab$ is called
      the {\em composition} of the {\em composable} pair $(a,b)$
      in $E^{(2)}(\mathcal{Y})$);
\item we have $i(ab)=i(b)$ and $t(ab)=t(a)$ for a composable pair $(a,b)$ in $E^{(2)}(\mathcal{Y})$;
\item the composition law is {\em associative}, that is, the composition
      $a(bc)$ coincides with $(ab)c$ for composable pairs $(a,b)$ and $(b,c)$ in $E^{(2)}(\mathcal{Y})$;
\item the elements $i(a)$ and $t(a)$ are distinct for each $a$ in $E(\mathcal{Y})$.
\end{enumerate}
Here $E^{(k)}(\mathcal{Y})$ denotes the set of
$k$\nobreakdash-sequences $(a_1, \dotsc, a_k)$ of elements of
$E(\mathcal{Y})$ satisfying 
the equality $i(a_j)=t(a_{j+1})$ for each $1\leq j\leq k-1$. 
Elements of $V(\mathcal{Y})$ are called {\em vertices} of
$\mathcal{Y}$, and those of $E(\mathcal{Y})$ are called (oriented) {\em
edges} of $\mathcal{Y}$. For an edge $a$ of $\mathcal{Y}$, the vertices 
$i(a)$ and $t(a)$ are respectively called the {\em initial} and {\em terminal}
vertices of $a$. A scwol $\mathcal{Y}$ has its {\em geometric realisation}
$|\mathcal{Y}|$ defined in an appropriate way, which is a (polyhedral)
complex all of whose cells are simplices; see
\cite[Chapter~III.$\mathcal{C}$ Section~1.3]{BH} for details.
A {\em morphism} $f \colon \mathcal{Y} \rightarrow \mathcal{Y}'$ of scwols 
consists of 
(set-theoretical) maps $V(\mathcal{Y})\rightarrow V(\mathcal{Y}')$ and 
$E(\mathcal{Y})\rightarrow E(\mathcal{Y}')$ which are compatible with the
scwol structures of $\mathcal{Y}$ and $\mathcal{Y}'$ (refer to
\cite[Chapter~III.$\mathcal{C}$ Section~1.5]{BH} for the precise definition).

In order to introduce the fundamental group of a scwol, we here 
summarise basic notion on edge paths.
Let $E^+(\mathcal{Y})$ (resp.\ $E^-(\mathcal{Y})$) denote the set 
consisting of an element denoted by $a^+$ (resp.\ $a^-$) for each edge
$a$ of $\mathcal{Y}$, whose initial and terminal vertices are determined
by $i(a^+)=t(a)$ and $t(a^+)=i(a)$ (resp.\ $i(a^-)=i(a)$ and
$t(a^-)=t(a)$). We denote by $E^\pm(\mathcal{Y})$ the disjoint union 
of $E^+(\mathcal{Y})$ and $E^-(\mathcal{Y})$. We also set
$(a^\pm)^{-1}=a^\mp$ (double sign in the same order). An {\em edge path} in
$\mathcal{Y}$ is a finite sequence $l=(e_1, \dotsc, e_n)$ of 
elements of $E^\pm(\mathcal{Y})$ satisfying $t(e_j)=i(e_{j+1})$ for 
each $1\leq j\leq n-1$. The vertices $i(e_1)$ and $t(e_n)$ are called 
the {\em initial} and {\em terminal} vertices of the edge path $l$, and 
denoted by $i(l)$ and $t(l)$ respectively. We define the {\em
concatenation} $l*l'$ of edge paths $l=(e_1, \dotsc, e_m)$ and $l'=(e_1',
\dotsc, e_n')$ by $l*l'=(e_1, \dotsc, e_m, e_1', \dotsc, e_n')$ when 
the pair $(l,l')$ satisfies $t(l)=i(l')$, and we also define the {\em
inverse edge path} $l^{-1}$ of an edge path $l=(e_1, \dotsc, e_n)$ by 
$l^{-1}=(e_n^{-1}, e_{n-1}^{-1}, \dotsc , e_1^{-1})$. 
An edge path is called an {\em
edge loop} when its initial vertex coincides with its terminal vertex; 
in the case its initial (and hence also terminal) vertex is called its
{\em base vertex}. 

Now let $\mathcal{Y}$ be a {\em connected} scwol (in the sense that arbitrary 
two vertices of $\mathcal{Y}$ are connected by an edge path in
$\mathcal{Y}$) 
and let us consider the set of all edge loops with base vertex $\sigma_0$. 
We endow this set with an equivalence relation $\sim$ (called {\em homotopy
equivalence}) generated by the following two elementary relations:
\begin{enumerate}[label=\roman*), widest=ii]
\item $(e_1, \dotsc, e_{j-1}, e_j, e_{j+1}, e_{j+2}, \dotsc , e_n) \sim
      (e_1, \dotsc, e_{j-1}, e_{j+2}, \dotsc, e_n)$ if $e_{j+1}$
      coincides with $e_j^{-1}$;
\item if $(a,b)$ is a composable pair in $E^{(2)}(\mathcal{Y})$, we impose
\begin{align*}
(e_1, \dotsc, e_{i-1}, e_i=a^+, e_{i+1}=b^+, e_{i+2}, \dotsc, e_m) \sim 
(e_1, \dotsc, e_{i-1}, (ab)^+, e_{i+2}, \dotsc, e_m)
\end{align*}
and
\begin{align*}
(e_1, \dotsc, e_{j-1}, e_j=b^-, e_{j+1}=a^-, e_{j+2}, \dotsc, e_n) \sim 
(e_1, \dotsc, e_{j-1}, (ab)^-, e_{j+2}, \dotsc, e_n).
\end{align*}
\end{enumerate}
The set $\pi_1(\mathcal{Y}, \sigma_0)$ of all (homotopy) equivalence classes 
of edge loops with base vertex $\sigma_0$  
is indeed equipped with a group structure whose group law is induced 
by the concatenation of edge loops: $[c]*[c']=[c*c']$. For an edge loop $c$, 
the inverse of $[c]$ is given 
by the homotopy class $[c^{-1}]$ of the inverse loop $c^{-1}$ of $c$, 
and the unit element is given by the homotopy class of the 
{\em constant loop $[c_{\sigma_0}]$ at $\sigma_0$}  (by definition 
the constant loop $c_{\sigma_0}$ corresponds to 
the ``empty word'' $c_{\sigma_0}=(\,)$ of $E^\pm (\mathcal{Y})$, both of 
whose initial and terminal vertices are defined as $\sigma_0$). 
We call $\pi_1(\mathcal{Y}, \sigma_0)$ the {\em fundamental group} of
the scwol $\mathcal{Y}$ at $\sigma_0$. 
We can construct an isomorphism $\pi_1(\mathcal{Y},
\sigma_0)\xrightarrow{\sim} \pi_1(\mathcal{Y}, \sigma_0'); [c]\mapsto
[l_{\sigma_0, \sigma_0'}^{-1}*c*l_{\sigma_0, \sigma_0'}]$ as in the case
of usual fundamental groups, where 
$l_{\sigma_0,\sigma_0'}$ is an edge path with initial vertex $\sigma_0$
and terminal vertex $\sigma_0'$ (since we assume that $\mathcal{Y}$ is connected, such $l_{\sigma_0,\sigma_0'}$ always exists; obviously this isomorphism is not a
canonical one). 

We end this subsection by quoting the following classical fact. 

\begin{prop} \label{prop:fundamental} 
Let $\mathcal{Y}$ be a connected scwol and $\sigma_0$ a vertex of $\mathcal{Y}$.
Then the fundamental group $\pi_1(\mathcal{Y}, \sigma_0)$ of the scwol
 $\mathcal{Y}$ is canonically isomorphic to the fundamental group 
$($in the usual sense$)$ 
$\pi_1(|\mathcal{Y}|, \sigma_0)$ of its geometric realisation. 
In particular, a connected scwol $\mathcal{Y}$ is 
{\em simply connected} $($in the sense that its
fundamental group is trivial$)$ if and only if its geometric realisation
$|\mathcal{Y}|$ is simply connected in the usual sense. 
\end{prop}

For details, see \cite[Chapter~III.$\mathcal{C}$ Section~1.8]{BH} and
\cite{Massey}.

\subsection{Complexes of groups and their fundamental groups} \label{ssc:FG}

A {\em complex of groups} $G(\mathcal{Y})$ over a scwol $\mathcal{Y}$ 
consists of three types of data: a group $G_\sigma$ for each vertex
$\sigma$ of $\mathcal{Y}$ called the {\em local group} at $\sigma$, 
an injective group homomorphism $\psi_a \colon G_{i(a)} \rightarrow
G_{t(a)}$ for each edge $a$ of $\mathcal{Y}$, and a specific element $g_{a,b}$ 
of $G_{t(a)}$, called a {\em twisting element}, for each composable 
pair $(a,b)$ in $E^{(2)}(\mathcal{Y})$. We impose the following two
constraints on these data:
\begin{itemize}
\item[-] (twisted commutativity)
	 the equality $g_{a,b}\psi_{ab}(x)g_{a,b}^{-1}=\psi_a\circ
	 \psi_b(x)$ holds for
	 each composable pair $(a,b)$ in $E^{(2)}(\mathcal{Y})$ and 
	 every element $x$ of $G_{i(b)}$;
\item[-] (cocycle condition) the equality 
	 $\psi_a(g_{b,c})g_{a,bc}=g_{a,b}g_{ab,c}$ holds
	 for each pairwisely composable triple $(a,b,c)$ in
	 $E^{(3)}(\mathcal{Y})$.
\end{itemize}
A complex of groups $G(\mathcal{Y})$ over $\mathcal{Y}$ is called {\em
simple} if all the twisting elements are trivial, that is, the element 
$g_{a,b}$ equals
the unit of $G_{t(a)}$ for
each composable pair $(a,b)$ in $E^{(2)}(\mathcal{Y})$. 

\begin{rem} \label{rem:cocycle}
We here remark that for a complex of groups {\em of dimension
at most $2$}, that is, for a complex of groups 
whose geometric realisation is of dimension at most $2$, 
the cocycle condition among twisting elements introduced above is 
just the empty condition (simply because $E^{(3)}(\mathcal{Y})$ is empty). 
Later we shall mainly study complexes of groups associated 
to essential tribranched surfaces in a $3$\nobreakdash-manifold, 
which we shall define in Section~\ref{ssc:cgtbs}. 
Obviously by construction they are of dimension at most $2$, 
and hence we do not have to consider the cocycle conditions 
whenever we are concerned with complexes of groups 
associated to essential tribranched surfaces.
\end{rem}

Let $f\colon \mathcal{Y}\rightarrow \mathcal{Y}'$ be a morphism of
scwols, and let $G(\mathcal{Y})$ and $G(\mathcal{Y}')$ be complexes 
of groups over $\mathcal{Y}$ and $\mathcal{Y'}$ respectively. 
A {\em morphism of complexes of groups} (over $f$) $\phi \colon
G(\mathcal{Y})\rightarrow G(\mathcal{Y}')$ consists of two types of
data: a group homomorphism between local groups 
$\phi_\sigma \colon
G_\sigma \rightarrow G_{f(\sigma)}$ for each vertex $\sigma$ of
$\mathcal{Y}$ and a specific element $\phi(a)$ of $G_{t(f(a))}$, 
called a {\em twisting element}, for each edge $a$ of $\mathcal{Y}$. 
We impose the following two constraints on these data:
\begin{itemize}
\item[--] (twisted commutativity)
	 the element $\phi(a)\bigl(\psi_{f(a)}\circ
	 \phi_{i(a)}(x)\bigr)\phi(a)^{-1}$ of $G_{t(f(a))}$ 
	 coincides with $\phi_{t(a)}\circ \psi_a(x)$
	 for each edge $a$ of $\mathcal{Y}$ and each element $x$ of
	 $G_{i(a)}$; 
\item[--] (compatibility among twisting elements) 
	 the element $\phi_{t(a)}(g_{a,b})\phi(ab)$ of $G_{t(f(a))}$ 
	 coincides with 
	 $\phi(a)\psi_{f(a)}(\phi(b))g_{f(a),f(b)}$
	 for each composable pair $(a,b)$ in $E^{(2)}(\mathcal{Y})$.
\end{itemize}
A morphism $\phi$ of complexes of groups is said to be an {\em 
isomorphism} if its local group homomorphism $\phi_\sigma$ is 
an isomorphism for each vertex $\sigma$ of $\mathcal{Y}$.
By regarding an abstract group $G$ as a complex of groups over the trivial
scwol (that is, the scwol consisting of a single vertex)
whose local group at its unique vertex is $G$, we may also consider 
the notion of morphisms $G(\mathcal{Y})\rightarrow G$ from a complex
of groups $G(\mathcal{Y})$ to an abstract group $G$. 

Next we introduce the notion of the fundamental group of a complex of
groups. Let $G(\mathcal{Y})$ be a complex of groups over a scwol $\mathcal{Y}$. 
A {\em $G(\mathcal{Y})$-path} in $\mathcal{Y}$ is a finite sequence 
$l=(g_0, e_1, g_1, \dotsc, e_n, g_n)$ where $(e_1, \dotsc, e_n)$ is an
edge path in $\mathcal{Y}$, $g_0$ is an element of the local group
$G_{i(e_1)}$ at $i(e_1)$ and $g_j$ is an element of the local group
$G_{t(e_j)}$ at $t(e_j)$ for each $1\leq j\leq n$. For
$G(\mathcal{Y})$\nobreakdash-paths we define their initial and terminal 
vertices, concatenations and inverse paths similarly to those of edge paths as follows.
\begin{description}
\item[Initial and terminal vertices] for a
	   $G(\mathcal{Y})$\nobreakdash-path $l=(g_0, e_1, g_1, \dotsc, e_n,
	   g_n)$, set $i(l)=i(e_1)$ and $t(l)=t(e_n)$.
\item[Concatenation] for $G(\mathcal{Y})$\nobreakdash-paths $l=(g_0,
	   e_1, g_1, \dotsc, e_m, g_m)$ and $l'=(g'_0, e'_1, g'_1,
	   \dotsc, e'_n, g'_n)$ satisfying $t(l)=i(l')$, set 
	   $l*l'=(g_0, e_1, g_1, \dotsc, e_m, g_mg'_0, e'_1, g'_1,
	   \dotsc, e'_n, g'_n)$.
\item[Inverse path] for a $G(\mathcal{Y})$\nobreakdash-path
	   $l=(g_0,e_1,g_1, \dotsc, e_n, g_n)$, define its inverse
	   $G(\mathcal{Y})$\nobreakdash-path $l^{-1}$ as  $l^{-1}=(g_n^{-1},
	   e_n^{-1},g_{n-1}^{-1}, \dotsc, e_1^{-1}, g_0^{-1})$.
\end{description}
Now let $FG(\mathcal{Y})$ be the {\em universal group associated to
$G(\mathcal{Y})$} which is defined by the following generators and
relations.
\begin{description}
\item[Generators] elements of all local groups $G_\sigma$ and elements
	   of $E^\pm (\mathcal{Y})$. 
\item[Relations] we impose on the generators the following four types of
	   relations: 
\begin{itemize}
\item[--] the group relations for each $G_\sigma$; 
\item[--] $(a^\pm)^{-1}=a^\mp$ for each edge $a$ in $\mathcal{Y}$ (double
	 sign in the same order);
\item[--] $a^+b^+=g_{a,b}(ab)^+$ for each composable pair $(a,b)$ in
	 $E^{(2)}(\mathcal{Y})$; 
\item[--] $\psi_a(x)=a^+xa^-$ for each edge $a$ of $\mathcal{Y}$ and each
	 element $x$ of $G_{i(a)}$.
\end{itemize}
\end{description}
Then it is easy to check that the morphism 
$\iota \colon G(\mathcal{Y})\rightarrow FG(\mathcal{Y})$, which consists of 
a group homomorphism $\iota_\sigma \colon G_\sigma \rightarrow
FG(\mathcal{Y}); g\mapsto g$ for each vertex $\sigma$ of $\mathcal{Y}$ 
and a twisting element $\iota(a)=a^+$ for each edge $a$ of $\mathcal{Y}$,
has a universal property among morphisms from $G(\mathcal{Y})$ to 
abstract groups. More specifically, for every 
morphism $\phi \colon G(\mathcal{Y})\rightarrow G$ from $G(\mathcal{Y})$
to an abstract group
$G$ we obtain a unique group homomorphism $F\phi \colon FG(\mathcal{Y})\rightarrow
G$ which satisfies $\phi=F\phi
\circ \iota$ (see \cite[Chapter~III.$\mathcal{C}$ Section~3.2]{BH} for details).

We associate to each $G(\mathcal{Y})$-loop $c=(g_0, e_1, g_1, \dotsc ,
e_n, g_n)$ with base vertex $\sigma_0$ 
an element $[c]$ of $FG(\mathcal{Y})$ which is 
by definition the image of the word $g_0e_1g_1\cdots e_ng_n$ in
$FG(\mathcal{Y})$. The image of $[~\cdot~]$ (as a map from the set of 
$G(\mathcal{Y})$\nobreakdash-loops with base vertex $\sigma_0$) 
is equipped with a group structure 
induced by concatenations, which we denote by $\pi_1(G(\mathcal{Y}),
\sigma_0)$ and call the {\em fundamental group} of $G(\mathcal{Y})$.
We remark that the definition of the fundamental group
$\pi_1(G(\mathcal{Y}), \sigma_0)$ of a complex of groups
$G(\mathcal{Y})$ (of higher dimension) introduced here 
is a direct generalisation of 
the definition of the fundamental group of a graph of groups 
due to Bass and Serre \cite[Section~5.1]{Serre}.

\subsection{Group actions on scwols and developability}

Let $\mathcal{X}$ be a scwol and $G$ an abstract group.
An {\em action of $G$ on $\mathcal{X}$} is a group homomorphism 
$G \rightarrow \mathrm{Aut}(\mathcal{X})$ satisfying the following two
conditions: 
\begin{enumerate}[label=\roman*), widest=ii]
\item the vertex $g.i(a)$ does not equal $t(a)$ for each edge $a$ of
      $\mathcal{X}$ and each element $g$ of $G$;
\item if $g.i(a)=i(a)$ holds for an edge $a$ of $\mathcal{X}$ and an
      element $g$ of $G$, we have $g.a=a$.
\end{enumerate}
Here $g.\sigma$ (resp.\ $g.a$) denotes the image of 
a vertex $\sigma$ of $\mathcal{X}$ (resp.\ an edge $a$ of
$\mathcal{X}$) under the automorphism of $\mathcal{X}$ induced by $g$. 
For such an action of $G$ on $\mathcal{X}$, 
we may construct the {\em quotient scwol} $\mathcal{Y}=G \backslash \mathcal{X}$ 
by setting $V(\mathcal{Y})=G \backslash V(\mathcal{X})$ and 
$E(\mathcal{Y})=G \backslash E(\mathcal{X})$ (initial and terminal
vertices and compositions in $\mathcal{Y}$ are determined 
in obvious manners; that is, 
$i(G.a)$ and $t(G.a)$ are defined as $G.i(a)$ and $G.t(a)$ respectively 
for an edge $a$ of $\mathcal{X}$,
and the composition $(G.a)(G.b)$ is defined as $G.ab$ for each composable pair $(a,b)$ in
$E^{(2)}(\mathcal{X})$).   

We may endow the quotient scwol $\mathcal{Y}=G \backslash \mathcal{X}$
with the structure of a complex of groups in the following way. For each 
vertex $\sigma$ of $\mathcal{Y}$, we choose a lift $\tilde{\sigma}$ of
$\sigma$ to $\mathcal{X}$ (that is, $\tilde{\sigma}$ is a vertex of
$\mathcal{X}$ whose $G$-orbit coincides with $\sigma$). The condition
ii) of the group action on a scwol implies that for each edge $a$ of 
$\mathcal{Y}$ with initial vertex $\sigma$, there exists a unique lift 
$\tilde{a}$ of $a$ to $\mathcal{X}$ with initial vertex
$\tilde{\sigma}$. Let us choose an element $h_a$ of $G$ satisfying 
$h_a.t(\tilde{a})=\widetilde{t(a)}$. We define the local group
$G_\sigma$ at a vertex $\sigma$ of $\mathcal{Y}$ as 
the isotropy subgroup $G_{\tilde{\sigma}}$ of $G$ at $\tilde{\sigma}$ with
respect to the group action of $G$ on $\mathcal{X}$. For each edge $a$
of $\mathcal{Y}$, we define a group homomorphism $\psi_a \colon
G_{i(a)}\rightarrow G_{t(a)}$ by $\psi_a(g)=h_agh_a^{-1}$.
Finally for each composable pair $(a,b)$ in $E^{(2)}(\mathcal{Y})$, 
we define a twisting element $g_{a,b}$ 
as $h_ah_bh_{ab}^{-1}$.
It is easy to verify that these data determine the structure of 
a complex of groups $G(\mathcal{Y})$ over the quotient scwol
$\mathcal{Y}$, which we call the {\em complex of groups associated to 
the group action of $G$ on $\mathcal{X}$}. Note that 
if we choose a different lift $\tilde{\sigma}'$ for each vertex $\sigma$
of $\mathcal{Y}$ and a different element $h_a'$ for each edge $a$ of
$\mathcal{Y}$, the resultant complex of groups $G'(\mathcal{Y})$ is 
still isomorphic to $G(\mathcal{Y})$; see
\cite[Chapter~III.$\mathcal{C}$ Section~2.9 (2)]{BH} for details.
When a complex of groups $G(\mathcal{Y})$ is associated to an action 
of $G$ on a scwol $\mathcal{X}$ as above, 
we may construct a morphism $\phi \colon G(\mathcal{Y})\rightarrow
G$ by setting $\phi_\sigma(g)=g$ for each vertex $\sigma$ and
$\phi(a)=h_a$ for each edge $a$.

A complex of groups $G(\mathcal{Y})$ over a scwol $\mathcal{Y}$ is
called {\em developable} if there exists a scwol $\mathcal{X}$ equipped
with an action of a group $G$ such that $G(\mathcal{Y})$ is isomorphic to 
the complex of groups associated to the group action of $G$ on
$\mathcal{X}$. Unlike graphs of groups, complexes of groups of higher
dimension are not always developable. The following proposition 
proposes a necessary and
sufficient condition for a complex of groups to be developable.

\begin{prop}[{\cite[Chapter~III.$\mathcal{C}$ Corollary~2.15]{BH}}]
 \label{prop:development} 
 A complex of groups $G(\mathcal{Y})$ is developable if and only if 
there exists a morphism from $G(\mathcal{Y})$ to a certain
 $($abstract$)$ group $G$ 
which is injective on each local group $G_\sigma$ of $G(\mathcal{Y})$.
\end{prop}

In fact if $G(\mathcal{Y})$ admits a morphism $\phi \colon 
G(\mathcal{Y})\rightarrow G$ which is injective 
on each local group of $G(\mathcal{Y})$, we may construct in a
canonical manner
a scwol $D(\mathcal{Y}, \phi)$ equipped with a group action of $G$ 
 (which is called the {\em development of $\mathcal{Y}$ with
respect to $\phi$}) by setting
\begin{align*}
V(D(\mathcal{Y}, \phi))&=\{ \, (g\phi_\sigma(G_\sigma), \sigma) \mid \sigma
 \in V(\mathcal{Y}), g\phi_\sigma(G_\sigma) \in
 G / \phi_\sigma(G_\sigma) \, \}, \\
E(D(\mathcal{Y}, \phi))&=\{ \, (g\phi_{i(a)}(G_{i(a)}), a) \mid a
 \in E(\mathcal{Y}), g\phi_{i(a)}(G_{i(a)}) \in
 G / \phi_{i(a)}(G_{i(a)}) \, \} 
\end{align*}
and
\begin{align*}
i((g\phi_{i(a)}(G_{i(a)}), a))&=(g\phi_{i(a)}(G_{i(a)}), i(a)), \\
t((g\phi_{i(a)}(G_{i(a)}), a))&=(g\phi(a)^{-1}\phi_{t(a)}(G_{t(a)}), t(a))
\end{align*}
for each $(g\phi_{i(a)}(G_{i(a)}), a)$ in $E(D(\mathcal{Y}, \phi))$.
The group $G$ acts on $D(\mathcal{Y},\phi)$ in a natural way; namely, 
an element $x$ of $G$ acts as 
\begin{align*}
x.(g\phi_\sigma(G_\sigma), \sigma)&=(xg\phi_\sigma(G_\sigma), \sigma),
 & x.(g\phi_{i(a)}(G_{i(a)}), a)&=(xg\phi_{i(a)}(G_{i(a)}),
 a).
\end{align*}
For details, see \cite[Chapter~III.$\mathcal{C}$ Theorem~2.13]{BH}.

\section{Tribranched surfaces and complexes of groups} \label{sc:TBS}

We introduce in this section the notion of {\em tribranched
surfaces} and {\em essential tribranched surfaces} which shall 
play key roles throughout this article. 
It is a certain generalisation of the concepts of surfaces (contained in a
$3$\nobreakdash-manifold) and essential surfaces (see, for example, 
\cite[Definition~1.5.1]{Shalen} for the definition of
essential surfaces).
After proposing the definitions of 
tribranched surfaces and essential tribranched surfaces 
in Section~\ref{ssc:TBS}, 
we observe that 
essential tribranched surfaces behave compatibly with 
the theory of complexes of groups in Sections~\ref{ssc:cgtbs} and
\ref{ssc:SETBS} (as 
the notion of essential surfaces is well adapted 
to Bass and Serre's theory on graphs of groups \cite{Serre}
in the original work of Culler and Shalen \cite{CS} ). 

\subsection{Tribranched surfaces and essential tribranched surfaces} \label{ssc:TBS}

Let $M$ be a $3$\nobreakdash-manifold with possibly nonempty boundary.
Let $\Sigma$ be a compact subset of $M$ such that the pair $(M, \Sigma)$ is
locally homeomorphic to $(\overline{\mathbb{H}}, Y \times [0, \infty))$,
where $\overline{\mathbb{H}}$ and $Y$ are defined by
\begin{align*}
\overline{\mathbb{H}} = \{\,  (z, s) \in \C \times \R \mid s \geq 0 \, \}, \qquad
Y = \{ \, r e^{\sqrt{-1} \theta} \in \C \mid r\in \R_{\geq 0} \text{ and
 }\theta = 0, \pm 2 \pi/3 \, \}.
\end{align*}
We denote by $C(\Sigma)$ the set of \textit{branched points} of $\Sigma$
corresponding to $\{0\} \times [0, \infty) \subset Y\times [0,\infty)$, by $S(\Sigma)$ the complement of a sufficiently small tubular neighbourhood of $C(\Sigma)$ in $\Sigma$, and by $M(\Sigma)$ the complement of a regular neighbourhood of $\Sigma$ in $M$.
The subsets $C(\Sigma)$ and $S(\Sigma)$ are a properly embedded
$1$-submanifold and a subsurface of $M$ respectively.
See Figure \ref{fig:TBS} for a local picture of $\Sigma$.

\begin{figure}[h]
\centering
\begin{tikzpicture}
\draw (0,0) rectangle (2.5,3);
\draw [very thick] (0,0) -- (0,3);
\draw (0,0) -- (-2,-1) -- (-2, 2) -- (0, 3);
\draw[dashed] (0,0) -- (-2,1);
\draw (-2,1) -- (-2.6,1.3) -- (-2.6,4.3) -- (0,3);

\draw (-3.2,3) -- (-4.7,-1.5) -- (2.8,-1.5) -- (4.2,3);

\draw [->,>=latex] (0.5,3.5) .. controls (0.3,3.55) and (0.1,3.5) .. (0,3.05);
\draw (0.5,3.5) node [right] {$C(\Sigma)$};
\draw (1.8,1.5) node {$S(\Sigma)$};
\draw (2,-1) node {$\partial M$};
\end{tikzpicture}
\caption{A tribranched surface $\Sigma$}
\label{fig:TBS}
\end{figure}
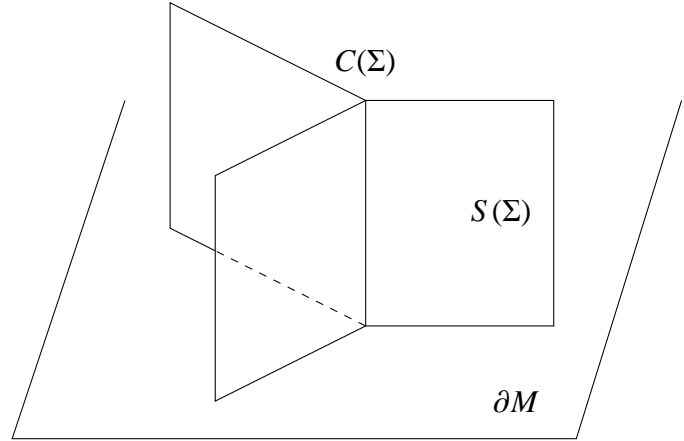  

\begin{defn}[Tribranched surfaces] \label{def:TBS}
Let $(M,\Sigma)$ be as above.
We call $\Sigma$ a {\em tribranched surface in $M$} if the following two conditions are satisfied:
\begin{enumerate}[labelindent=.5em, labelwidth=3.5em, labelsep*=1em,
 leftmargin=!, label=(TBS\arabic*)]
\item the intersection of $\Sigma$ and a sufficiently small tubular
      neighbourhood of $C(\Sigma)$ in $M$ is homeomorphic to $Y \times C(\Sigma)$;
\item the subsurface $S(\Sigma)$ is orientable.
\end{enumerate} 
\end{defn}

In the following, we will suppress the base point in the notation of
fundamental groups unless specifically noted.

\begin{defn}[Essential tribranched surfaces] \label{def:ETBS}
A tribranched surface $\Sigma$ in $M$ is said to be {\em essential} if the following three conditions are satisfied, other than the conditions (TBS1) and (TBS2) of Definition~\ref{def:TBS}:
\begin{enumerate}[labelindent=.5em, labelwidth=4em, labelsep*=1em,
 leftmargin=!, label=(ETBS\arabic*)]
\item for any component $N$ of $M(\Sigma)$, the homomorphism $\pi_1 (N)
      \to \pi_1(M)$ induced by the natural inclusion map
      $N\hookrightarrow M$ is not surjective; 
\item for any components $C$, $S$, $N$ of $C(\Sigma)$, $S(\Sigma)$, $M(\Sigma)$
      respectively, 
      if the homomorphisms $\pi_1 (C) \to \pi_1 (S)$ and $\pi_1 (S) \to \pi_1
      (N)$ are induced by  the natural
      inclusion maps, they are injective; 
\item no component of $\Sigma$ is contained in a $3$-ball in $M$ or a collar of $\partial M$.
\end{enumerate} 
\end{defn}

\begin{rem}
An essential surface (in the usual sense) in $M$ is regarded as an
 essential tribranched surface without any branched points. 
\end{rem}

\subsection{Complexes of groups associated to essential tribranched surfaces} \label{ssc:cgtbs}

It is well known that 
one may associate a {\em graph of groups} to an essential surface
 (without any branched points) embedded in a $3$-manifold $M$, which 
gives a {\em splitting} of the $3$\nobreakdash-manifold group $\pi_1(M)$; 
we refer the readers to \cite[Sections~1.4 and 1.5]{Shalen}. 
Then the concept of essential tribranched surfaces, which is a more general
notion including essential surfaces, should be 
closely related to the theory of {\em complexes of groups} of higher dimension.
Here we discuss the relation between them.

Now let $M$ be a $3$-manifold which is compact, connected, 
irreducible and orientable. Suppose that $M$ contains 
a tribranched surface $\Sigma$. 

\medskip
\paragraph{\bfseries The dual $2$-complex associated to $\Sigma$}
A cellular map between CW-complexes $f\colon X\rightarrow Y$ is 
said to be {\em combinatorial} if it maps each open cell of $X$
homeomorphically to an open cell of $Y$, and a CW-complex $X$ is said to
be {\em combinatorial} if, for each cell $e_\lambda$ of $X$ 
of dimension $n_\lambda$, the characteristic map 
$\varphi_\lambda \colon D^{n_\lambda} \rightarrow
X^{(n_\lambda-1)}$ of $e_\lambda$ is a combinatorial cellular map 
with respect to a certain cellular complex structure on the
$n_\lambda$\nobreakdash-dimensional closed unit ball $D^{n_\lambda}$. 
In this paragraph we associate to the pair $(M,\Sigma)$ 
a combinatorial CW-complex $Y_\Sigma=Y_{(M, \Sigma)}$ of dimension $2$. 
The construction of $Y_\Sigma$ which we shall explain below 
is a natural generalisation of a well-known construction 
of the {\em dual graph} of a bicollared surface 
contained in a $3$\nobreakdash-manifold. The readers are referred 
to the exposition \cite[Section~1.4]{Shalen}, for example, 
for details on the classical construction of dual graphs.

Recall that $C(\Sigma)$ denotes the set of branched points of $\Sigma$. 
Let $C$ be a connected component of $C(\Sigma)$, and 
let $D^2$ (resp.\ $\mathring{D}^2$) denote 
the closed unit disk $\{ \,z \in \mathbb{C} \mid
|z|\leq 1 \, \}$ (resp.\ the open unit disk $\{ \, z \in \mathbb{C} \mid
|z|<1 \, \}$). For each $C$, there exists a tubular neighbourhood 
$h_C \colon C \times D^2 \rightarrow M$ 
by virtue of the condition (TBS1) of tribranched surfaces; 
more specifically, $h_C$ induces a homeomorphism of $C\times
D^2$ onto a neighbourhood of $C$ in $M$ and satisfies $h_C(x,0)=x$ 
for each point $x$ of $C$. Furthermore $h\vert_{C\times
(\mathring{D}^2 \cap Y)}$ induces a homeomorphism 
of $C\times (\mathring{D}^2\cap Y)$ onto 
a regular neighbourhood of $C$ in $\Sigma$. We choose and fix 
such a tubular neighbourhood $h_C$ for each connected component $C$ 
of $C(\Sigma)$. We denote by $U_C$ (resp.\ $\bar{U}_C$) the 
open tubular neighbourhood $h_C(C\times \mathring{D}^2)$
(resp.\ the closed tubular neighbourhood $h_C(C\times D^2)$) of $C$ in $M$. 

Next let $S$ be an arbitrary connected 
component of $S(\Sigma)=\Sigma \setminus \bigcup_{C\in
\pi_0(C(\Sigma))} U_C$.  
The condition (TBS2) combined with the theory of 
regular neighbourhoods provides us with a homeomorphism 
$h_S \colon S\times [-1, 1] \rightarrow M$ 
onto a bicollar neighbourhood of $S$ 
in $M\setminus \bigcup_{C\in \pi_0(C(\Sigma))} U_C$; namely  
$h_S$ satisfies  $h_S(x,0)=x$ for each point $x$ of $S$ and 
$h_S(\partial S\times [-1,1])$ coincides with the intersection 
of $h_S(S\times [-1,1])$ and 
$\partial M \cup \bigcup_{C \in \pi_0(C(\Sigma))} \partial \bar{U}_C$.
We also choose and fix such a bicollar neighbourhood $h_S$ 
for each connected component $S$ of $S(\Sigma)$. 
We further assume that the closed sets
$h_S(S\times [-1,1])$ are pairwisely disjoint 
after replacing them by thinner ones if necessary.
We denote by $U_S$ (resp.\ $\bar{U}_S$) the subset $h_S(S \times
(-1,1))$ (resp.\ $h_S(S\times [-1,1])$) of $M$ which is an open (resp.\ a
closed) bicollar neighbourhood of $S$ in $M\setminus \bigcup_{C\in
\pi_0(C(\Sigma))} U_C$. 

We denote by $M(\Sigma)$ the complement of
$\bigcup_{C\in\pi_0(C(\Sigma))} U_C \cup \bigcup_{S \in
\pi_0(S(\Sigma))} U_S$ in $M$. 
Note that all of $\pi_0(C(\Sigma))$, $\pi_0(S(\Sigma))$ and
$\pi_0(M(\Sigma))$ are finite sets due to the compactness of $M$. 
We thus obtain a partition of $M$ into disjoint subsets:
\begin{align} \label{eq:disj}
M=\bigsqcup_{N \in \pi_0(M(\Sigma))} N \sqcup \underset{\substack{S \in
 \pi_0(S(\Sigma)) \\ t\in (-1,1)}}{\bigsqcup} h_S(S\times \{ t  \})
 \sqcup \underset{\substack{C
 \in \pi_0(C(\Sigma)) \\ s \in \mathring{D}^2}}{\bigsqcup} h_C(C \times
 \{  s \}).
\end{align}  
We use the notation $x\sim_\Sigma y$ 
to indicate that both of two points $x$ and $y$ of $M$  
are contained in one of the disjoint subsets 
occurring in the right hand side of (\ref{eq:disj}). Obviously 
$\sim_\Sigma$ defines an equivalence relation on $M$. Set
$Y_\Sigma=Y_{(M,\Sigma)}
=M/\! \! \sim_\Sigma$
and endow $Y_\Sigma$ with the quotient topology. 
One then easily observes that $Y_\Sigma$ 
is a combinatorial CW-complex of dimension $2$
whose $0$\nobreakdash-cells, $1$\nobreakdash-cells and 
$2$\nobreakdash-cells are labeled by elements of 
$\pi_0(M(\Sigma))$, $\pi_0(S(\Sigma))$ and $\pi_0(C(\Sigma))$
respectively. Moreover, for each $2$\nobreakdash-cell $e_C$ of
$Y_\Sigma$, the characteristic map $\varphi_C \colon D^2 \rightarrow 
Y_\Sigma^{(1)}$ is a combinatorial cellular map 
with respect to the following cellular complex structure on $D^2$:
\begin{align*}
\biggl. D^2=\mathring{D}^2 \sqcup \bigsqcup_{a=1}^3 \biggl\{\, e^{\sqrt{-1}\theta}
\; \biggr| \; \frac{2}{3}(a-1)\pi<\theta<\frac{2}{3}a\pi \, \biggr\} \sqcup
 \bigsqcup_{a=1}^3 \left\{ e^{\frac{2}{3}a\pi\sqrt{-1}} \right\}.
\end{align*}
Roughly speaking, this implies that each $2$\nobreakdash-cell of $Y_\Sigma$ 
may be identified with a $2$\nobreakdash-simplex whose boundaries
are appropriately glued to the $1$\nobreakdash-skeleton
$Y_\Sigma^{(1)}$ of $Y_\Sigma$. It is straightforward to check that a $1$\nobreakdash-cell $e_S$ 
(labeled by an element $S$ of
$\pi_0(S(\Sigma))$) occurs in the boundary of a $2$\nobreakdash-cell 
$e_C$ (labeled by an element $C$ of $\pi_0(C(\Sigma))$) if and only if 
the intersection of $\bar{U}_S$ and $\bar{U}_C$ is nonempty. 
Similarly a $0$\nobreakdash-cell $e_N$ (labeled by an element $N$ of
$\pi_0(M(\Sigma))$) occurs in the boundary of a $1$\nobreakdash-cell
$e_S$ (resp.\ a $2$\nobreakdash-cell $e_C$) if and only if the
intersection of $N$ and $\bar{U}_S$
(resp.\ $\bar{U}_C$) is nonempty. 
We call $Y_\Sigma$ the {\em dual {\upshape (2-)}complex} associated to the
tribranched surface $\Sigma$.
Figure~\ref{fig:dual} illustrates a
local picture of the dual $2$-complex $Y_\Sigma$ associated to a
tribranched surface $\Sigma$.

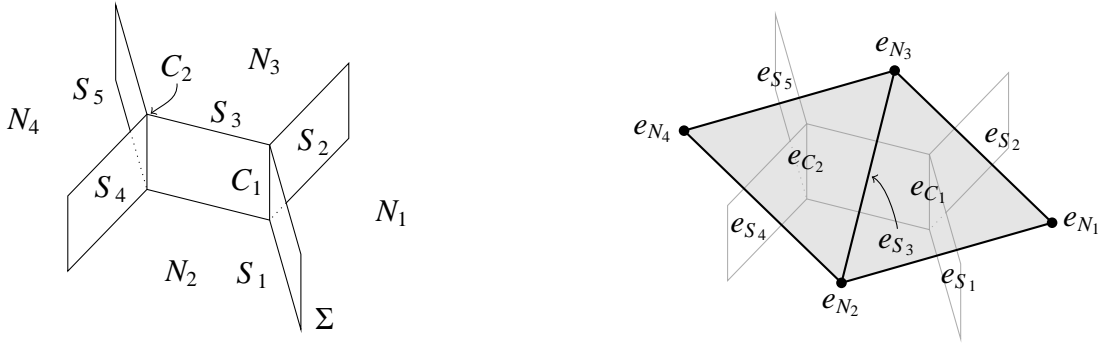
\begin{figure}[ht]
\centering
\begin{tabular}{ccc}
\begin{tikzpicture}
\coordinate (O) at (0,0);
\coordinate (A1) at (2.8,0.8);
\coordinate (P) at ($(O)!1!60:(A1)$);
\coordinate (A2) at ($(O)!1!60:(P)$);
\coordinate (G1) at ($1/3*(O)+1/3*(A1)+1/3*(P)$);
\coordinate (G2) at ($1/3*(O)+1/3*(A2)+1/3*(P)$);
\coordinate (M1) at ($(O)!.5!(A1)$);
\coordinate (M2) at ($(O)!.5!(A2)$);
\coordinate (N1) at ($(P)!.5!(A1)$);
\coordinate (N2) at ($(P)!.5!(A2)$);
\coordinate (C2) at ($(P-|G2)+(0.4,-0.1)$);

\path [name path=L1] ($(G1)-(0,0.5)$) -- ($(G1)!1.8!(N1)-(0,0.5)$);
\path [name path=L2] ($(G2)-(0,0.5)$) -- ($(G2)!1.8!(N2)-(0,0.5)$);
\path [name path=l1] ($(G1)+(0,0.5)$) -- ($(G1)!1.8!(M1)+(0,0.5)$);
\path [name path=l2] ($(G2)+(0,0.5)$) -- ($(G2)!1.8!(M2)+(0,0.5)$);
\path [name intersections={of=L1 and l1}];
\coordinate (Z1) at (intersection-1);
\path [name intersections={of=L2 and l2}];
\coordinate (Z2) at (intersection-1);

\draw ($(G1)+(0,0.5)$) -- ($(G2)+(0,0.5)$) -- ($(G2)+(0,-0.5)$) -- ($(G1)+(0,-0.5)$) --cycle;
\draw ($(G1)+(0,0.5)$) -- ($(G1)!1.8!(M1)+(0,0.5)$) -- ($(G1)!1.8!(M1)+(0,-0.5)$) -- ($(G1)+(0,-0.5)$);
\draw ($(G1)+(0,0.5)$) -- ($(G1)!1.8!(N1)+(0,0.5)$) -- ($(G1)!1.8!(N1)+(0,-0.5)$) -- (Z1);
\draw [dotted] (Z1) -- ($(G1)-(0,0.5)$);
\draw ($(G2)+(0,0.5)$) -- ($(G2)!1.8!(M2)+(0,0.5)$) -- ($(G2)!1.8!(M2)+(0,-0.5)$) -- ($(G2)+(0,-0.5)$);
\draw ($(G2)+(0,0.5)$) -- ($(G2)!1.8!(N2)+(0,0.5)$) -- ($(G2)!1.8!(N2)+(0,-0.5)$) -- (Z2);
\draw [dotted] (Z2) -- ($(G2)+(0,-0.5)$);

\node at (1.9,-0.6) {$\Sigma$};
\node at (G1) [left,xshift=2pt] {$C_1$};
\node at (C2) {$C_2$};
\node at (M1) [below, xshift=-13pt, yshift=-3pt] {$S_1$};
\node at (N1) [yshift=-2pt] {$S_2$};
\node at ($1/2*(P)$) [above, yshift=+11pt, xshift=7pt] {$S_3$};
\node at (M2) [xshift=3pt,yshift=3pt] {$S_4$};
\node at (N2) [left, xshift=-5pt] {$S_5$};
\node at (A1) {$N_1$};
\node at (O) {$N_2$};
\node at ($(P)+(0.4,0)$) {$N_3$};
\node at (A2) {$N_4$};

\draw [thin,->,>=latex] ($(C2)-(0,0.2)$) .. controls ($(C2)+(0.01,-0.4)$) and
 ($(G2)+(0.25,0.65)$) .. ($(G2)+(0.05,0.51)$); 
\end{tikzpicture}
& \hspace*{2cm} &
\begin{tikzpicture}
\coordinate (O) at (0,0);
\coordinate (A1) at (2.8,0.8);
\coordinate (P) at ($(O)!1!60:(A1)$);
\coordinate (A2) at ($(O)!1!60:(P)$);
\coordinate (G1) at ($1/3*(O)+1/3*(A1)+1/3*(P)$);
\coordinate (G2) at ($1/3*(O)+1/3*(A2)+1/3*(P)$);
\coordinate (M1) at ($(O)!.5!(A1)$);
\coordinate (M2) at ($(O)!.5!(A2)$);
\coordinate (N1) at ($(P)!.5!(A1)$);
\coordinate (N2) at ($(P)!.5!(A2)$);
\coordinate (S3) at ($7/11*(G1)$);

\path [name path=L1] ($(G1)-(0,0.5)$) -- ($(G1)!1.8!(N1)-(0,0.5)$);
\path [name path=L2] ($(G2)-(0,0.5)$) -- ($(G2)!1.8!(N2)-(0,0.5)$);
\path [name path=l1] ($(G1)+(0,0.5)$) -- ($(G1)!1.8!(M1)+(0,0.5)$);
\path [name path=l2] ($(G2)+(0,0.5)$) -- ($(G2)!1.8!(M2)+(0,0.5)$);
\path [name intersections={of=L1 and l1}];
\coordinate (Z1) at (intersection-1);
\path [name intersections={of=L2 and l2}];
\coordinate (Z2) at (intersection-1);

\fill [color=gray!50, opacity=.4] (O) -- (A1) -- (P) -- (A2) -- cycle;

\draw [color=gray!80, opacity=.8, thin] ($(G1)+(0,0.5)$) -- ($(G2)+(0,0.5)$) -- ($(G2)+(0,-0.5)$) -- ($(G1)+(0,-0.5)$) --cycle;
\draw [color=gray!80, opacity=.8,thin] ($(G1)+(0,0.5)$) -- ($(G1)!1.8!(M1)+(0,0.5)$) -- ($(G1)!1.8!(M1)+(0,-0.5)$) -- ($(G1)+(0,-0.5)$);
\draw [color=gray!80, opacity=.8,thin] ($(G1)+(0,0.5)$) -- ($(G1)!1.8!(N1)+(0,0.5)$) -- ($(G1)!1.8!(N1)+(0,-0.5)$) -- (Z1);
\draw [color=gray!80, opacity=.8,thin,dotted] (Z1) -- ($(G1)-(0,0.5)$);
\draw [color=gray!80, opacity=.8,thin] ($(G2)+(0,0.5)$) -- ($(G2)!1.8!(M2)+(0,0.5)$) -- ($(G2)!1.8!(M2)+(0,-0.5)$) -- ($(G2)+(0,-0.5)$);
\draw [color=gray!80, opacity=.8,thin] ($(G2)+(0,0.5)$) -- ($(G2)!1.8!(N2)+(0,0.5)$) -- ($(G2)!1.8!(N2)+(0,-0.5)$) -- (Z2);
\draw [color=gray!80, opacity=.8,thin, dotted] (Z2) -- ($(G2)+(0,-0.5)$);

\fill (A1) circle (2pt) node [right] {$e_{N_1}$};
\fill (O) circle (2pt) node [below] {$e_{N_2}$};
\fill (P) circle (2pt) node [above] {$e_{N_3}$};
\fill (A2) circle (2pt) node [left] {$e_{N_4}$};

\draw [thick] (O) -- (A1) -- (P) -- (A2) -- cycle;
\draw [thick] (O) -- (P);

\node at (M1) [below, xshift=5pt, yshift=-2pt] {$e_{S_1}$};
\node at (N1) [right, yshift=3pt] {$e_{S_2}$};
\node at (S3) [below, yshift=1pt] {$e_{S_3}$};
\node at (M2) [below, xshift=-5pt] {$e_{S_4}$};
\node at (N2) [above, xshift=-5pt] {$e_{S_5}$};  

\draw [thin,->,>=latex] ($(S3)-(0,0.05)$) .. controls ($(S3)+(-0.01,0.05)$) and
 ($1/2*(P)+(0.25,-0.1)$) .. ($1/2*(P)+(0.05,0)$);

\node at (G1) {$e_{C_1}$};
\node at (G2) {$e_{C_2}$};
\end{tikzpicture}
\end{tabular}
\caption{The dual 2-complex $Y_\Sigma$ associated to a tribranched
 surface $\Sigma$}
\label{fig:dual}
\end{figure}

\begin{rem}
The combinatorial CW-complex $Y_\Sigma$ is nothing but a
 ($M_0$-)polyhedral complex of dimension $2$ 
in the sense of \cite[Chapter~I.7, Definition~7.37]{BH} 
all of whose cells are (Euclidean) simplices. 
One often requires in many other references, however, 
that the intersection of two polytopes of a 
polyhedral complex should consist of a {\em single} common face of them 
unless it is empty. Therefore we here adopt the term a ``combinatorial
 CW-complex of dimension $2$'' rather than a ``polyhedral complex of
 dimension $2$'' in order to avoid terminological confusion. 
\end{rem}

\medskip
\paragraph{\bfseries The scwol associated to $\Sigma$} 
We next associate a scwol $\mathcal{Y}_\Sigma=\mathcal{Y}_{(M,\Sigma)}$
to the dual complex $Y_\Sigma$ in a canonical way (refer also to 
\cite[Chapter~III.$\mathcal{C}$, Section~1, Example~1.4~(2)]{BH}). 
By identifying each $2$\nobreakdash-cell 
of $Y_\Sigma$ with a $2$\nobreakdash-simplex, we may 
consider the {\em first barycentric subdivision} of the dual complex
$Y_\Sigma$ associated to $\Sigma$. We define the vertex set $V(\mathcal{Y}_\Sigma)$
of $\mathcal{Y}_\Sigma$ as the set of all cells of $Y_\Sigma$ 
(or equivalently, the set of the barycentres of all cells in
$Y_\Sigma$). Therefore every element of $V(\mathcal{Y}_\Sigma)$ 
is labeled by an element 
of the disjoint union of $\pi_0(C(\Sigma))$, $\pi_0(S(\Sigma))$ and
$\pi_0(M(\Sigma))$, which we denote by $\Lambda$ and regard as the index
set. We also define the edge set $E(\mathcal{Y}_\Sigma)$ 
of $\mathcal{Y}_\Sigma$ as the set of all $1$\nobreakdash-cells 
of the first barycentric subdivision of $Y_\Sigma$. 
One then readily observes that there exists an edge $a$ of
$\mathcal{Y}_\Sigma$ connecting two vertices $\sigma_\lambda$ and
$\sigma_\mu$ (labeled 
by elements $\lambda$ and $\mu$ of $\Lambda$ respectively) 
if and only if the cell $e_\mu$ of $Y_\Sigma$ labeled by $\mu$ 
occurs in the boundary of the cell $e_\lambda$ labeled by $\lambda$ 
or vice versa (in particular $\sigma_\lambda$ and $\sigma_\mu$ are
distinct). For an edge $a$ of $\mathcal{Y}_\Sigma$ connecting 
vertices $\sigma_\lambda$ and $\sigma_\mu$, 
we set $i(a)=\sigma_\lambda$ and $t(a)=\sigma_\mu$ if 
the cell $e_\mu$ of $Y_\Sigma$ corresponding to $\sigma_\mu$ 
occurs in the boundary of the cell $e_\lambda$ corresponding to
$\sigma_\lambda$.  
There exists a natural composition law among edges 
of $\mathcal{Y}_\Sigma$: namely $ab=c_{a,b}$ for each composable pair
$(a,b)$ in $E^{(2)}(\mathcal{Y}_\Sigma)$. Here $c_{a,b}$ denotes a unique edge
with $i(c_{a,b})=i(b)$ and $t(c_{a,b})=t(a)$ such that all of $a$, $b$
and $c_{a,b}$ occur in the boundary of a single $2$\nobreakdash-cell in
the first barycentric subdivision of $Y_\Sigma$ (see
Figure~\ref{fig:scwol} for details). 
Note that if a pair $(a, b)$ of edges of
$\mathcal{Y}_\Sigma$ is composable, $i(b)$ is labeled by 
an element of $\pi_0(C(\Sigma))$, $t(b)=i(a)$ is labeled by an
element of $\pi_0(S(\Sigma))$,
and $t(a)$ is labeled by an element of $\pi_0(M(\Sigma))$ respectively.
It is obvious that $\mathcal{Y}_\Sigma=(V(\mathcal{Y}_\Sigma),
E(\mathcal{Y}_\Sigma))$ equipped with the structures explained above 
satisfies all the conditions (Scw1)--(Scw4) of scwols 
(note that (Scw3) is now the empty condition). 

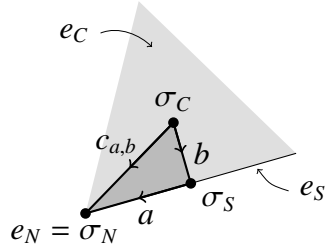
\begin{figure}[ht]
\centering
\begin{tikzpicture}
\coordinate (O) at (0,0) node [below] at (O) [xshift=-8pt] {$e_N=\sigma_N$};
\coordinate (A) at (2.8,0.8);
\coordinate (P) at ($(O)!1!60:(A)$);
\coordinate (M) at ($(O)!.5!(A)$);
\coordinate (G) at ($1/3*(O)+1/3*(A)+1/3*(P)$);

\coordinate (R) at ($1/8*(A)+3/4*(P)$);
\coordinate (T) at ($(O)!.8!(A)$);

\coordinate (C) at (0.2, 2.4) node at (C) [left] {$e_C$}; 
\coordinate (S) at (2.7, 0.3) node at (S) [right] {$e_S$};

\fill [color=gray!50, opacity=.5] (O) -- (A) -- (P) -- (O);
\fill [color=gray!60, opacity=.8] (O) -- (M) -- (G) -- (O);
\draw [thick] (O) -- (G) -- (M) -- cycle;
\draw [thick, ->,>=stealth] (G) -- ($(G)!.6!(M)$) node [right] {$b$};
\draw [thick, ->,>=stealth] (G) -- ($(G)!.5!(O)$) node [above, xshift=-5pt] {$c_{a,b}$};
\draw [thick, ->,>=stealth] (M) -- ($(M)!.5!(O)$) node [below, xshift=3pt] {$a$};
\draw (O) -- (A);

\node [right, yshift=-7pt] at (M) {$\sigma_S$};
\node [above] at (G) {$\sigma_C$};
\fill (O) circle (2pt) (G) circle (2pt) (M) circle (2pt);

\draw [->, ultra thin,>=latex] (C) .. controls ($(C)+(0.25,0.1)$) and ($(R)+(-0.15,0.15)$) .. (R);
\draw [->, ultra thin,>=latex] (S) .. controls ($(S)+(-0.25,0.08)$) and ($(T)+(0.2,-0.25)$) .. ($(T)+(0.05,-0.05)$); 
\end{tikzpicture}
\caption{The scwol structure of $\mathcal{Y}_\Sigma$ on a $2$-simplex of $Y_\Sigma$}
\label{fig:scwol}
\end{figure}

\medskip
\paragraph{\bfseries The complex of groups associated to $\Sigma$}
We now endow $\mathcal{Y}_\Sigma$ with the natural structure of a complex of groups.
Let us choose and fix a point $x_\lambda$ in $\lambda$ and 
define the local group $G^\Sigma_\lambda=G^\Sigma_{\sigma_\lambda}$ 
at $\sigma_\lambda$ 
as the fundamental group $\pi_1(\lambda, x_\lambda)$ (in the usual sense) for each element 
$\lambda$ of $\Lambda$; recall that the label set $\Lambda$ consists of 
connected subspaces of $M$. 
We next associate a group homomorphism 
$\psi^\Sigma_a \colon G^\Sigma_{i(a)} \rightarrow G^\Sigma_{t(a)}$ 
to each edge $a$. 
Let $\lambda$ and $\mu$ be elements of $\Lambda$ satisfying
$i(a)=\sigma_\lambda$ and $t(a)=\sigma_\mu$.
The existence of the edge $a$ implies that 
the cell $e_\mu$ of $Y_\Sigma$ corresponding to $\sigma_\mu$ occurs 
in the boundary of the cell $e_\lambda$ corresponding to $\sigma_\lambda$, 
and in particular the intersection of $\bar{U}_\lambda$
and $\bar{U}_{\mu}$ is nonempty as we have already remarked 
(with the convention $U_N=\bar{U}_N=N$ for 
each element $N$ of $\pi_0(M(\Sigma))$). 
We may thus 
take a path $l_{\lambda, \mu} \colon [0,1] \rightarrow \bar{U}_\lambda \cup \bar{U}_\mu$
satisfying $l_{\lambda, \mu}(0)=x_\lambda$ and $l_{\lambda, \mu}(1)=x_\mu$. We choose and fix such a path
$l_{\lambda, \mu}$ for each edge $a$ with $i(a)=\sigma_\lambda$ and
$t(a)=\sigma_\mu$.  
We may readily verify that $\mu$ is a deformation
retract of $\bar{U}_\lambda \cup \bar{U}_\mu$ 
by the definition of $\bar{U}_\lambda$  
as a tubular or bicollar neighbourhood, and therefore
 we may define a group homomorphism $\psi^\Sigma_a \colon
 G^\Sigma_\lambda \rightarrow G^\Sigma_\mu$ as the composition
\begin{align*}
G^\Sigma_\lambda =\pi_1(\lambda, x_\lambda) \rightarrow \pi_1(\bar{U}_\lambda \cup
 \bar{U}_\mu, x_\lambda) \xrightarrow{(\sharp)} \pi_1(\bar{U}_\lambda \cup \bar{U}_\mu, x_\mu)
 \xrightarrow{\sim} \pi_1(\mu, x_\mu)=G^\Sigma_\mu
\end{align*}
where the first map is induced from the natural inclusion $\lambda
\hookrightarrow \bar{U}_\lambda \cup \bar{U}_\mu$ and the last isomorphism is 
induced from a deformation retraction from $\bar{U}_\lambda \cup \bar{U}_\mu$ to
$\mu$. The middle map $(\sharp)$ is the change of base points with
respect to the path $l_{\lambda, \mu}$, or in other words, the map defined 
by $[c]\mapsto [l_{\lambda, \mu}^{-1}cl_{\lambda, \mu}]$.
Here we define the {\em concatenation} $l_1l_2$ of two paths $l_1, l_2
\colon [0,1]\rightarrow X$ in a topological space $X$ with $l_1(1) =
l_2(0)$ as follows: 
\begin{align*}
l_1l_2 (t) =\begin{cases}
l_1(2t) & \text{for } 0\leq t\leq \frac{1}{2}, \\
l_2(2t-1) & \text{for } \frac{1}{2}\leq t\leq 1.
\end{cases}
\end{align*}
We finally define a twisting element $g^\Sigma_{a,b}$ for each composable pair 
$(a, b)$ in $E^{(2)}(\mathcal{Y}_\Sigma)$. Suppose that the vertices $i(b)$,
$t(b)(=i(a))$ and $t(a)$ are labeled by elements $C$ of
$\pi_0(C(\Sigma))$, $S$ of $\pi_0(S(\Sigma))$ and $N$ of
$\pi_0(M(\Sigma))$ respectively. Then we define $g^\Sigma_{a,b}$ 
as the image of 
$[l_{S, N}^{-1}l_{C, S}^{-1}l_{C, N}]$ under the map $\pi_1(N\cup \bar{U}_S\cup \bar{U}_C,
x_N) \rightarrow \pi_1(N, x_N)=G^\Sigma_N$ induced by a deformation
retraction from $N\cup \bar{U}_S\cup \bar{U}_C$ to $N$. 
The twisted commutativity 
\begin{align} \label{eq:TC}
g^\Sigma_{a,b}\psi^\Sigma_{ab}([c])(g_{a,b}^{\Sigma})^{-1}=\psi^\Sigma_{a} \circ \psi^\Sigma_{b}([c])
\end{align}
straightforwardly holds for each element $[c]$ of $G^\Sigma_C=\pi_1(C, x_C)$.
We have now verified, combining Remark~\ref{rem:cocycle} with the
calculations above, that  
$G(\mathcal{Y}_\Sigma)=(\mathcal{Y}_\Sigma, \{ \psi^\Sigma_a \}_{a\in E(\mathcal{Y}_\Sigma)},
\{ g^\Sigma_{a,b} \}_{(a,b)\in E^{(2)}(\mathcal{Y}_\Sigma)})$ satisfies 
all the conditions of complexes of groups 
over $\mathcal{Y}_\Sigma$ except for injectivity of each $\psi^\Sigma_a$. 
If we further assume that the tribranched surface $\Sigma$ under
consideration is {\em essential}, we readily observe that 
every $\psi^\Sigma_a$ is injective due to the condition (ETBS2) and the twisted
commutativity (\ref{eq:TC}).
As a consequence, the triple $G(\mathcal{Y}_\Sigma)$ is indeed a $2$-complex 
of groups over the scwol $\mathcal{Y}_\Sigma$ when $\Sigma$ is
essential, which we call the {\em complex of groups associated to 
the essential tribranched surface $\Sigma$}.

\medskip
Let us choose and fix a point $x_0$ in
$M(\Sigma)$ and a path $l_\lambda \colon [0,1]\rightarrow M$ 
for each element $\lambda$ of $\Lambda$ such that $l_\lambda(0)=x_0$ and
$l_\lambda(1)=x_\lambda$. We define a morphism $\phi_\Sigma \colon
G(\mathcal{Y}_\Sigma)\rightarrow \pi_1(M,x_0)$ as follows. 
For each label $\lambda$, we define a group homomorphism $\phi_{\Sigma,\lambda
}\colon G^\Sigma_\lambda\rightarrow \pi_1(M,x_0)$ as the composition
\begin{align} \label{eq:identify}
G^\Sigma_\lambda=\pi_1(\lambda, x_\lambda) \rightarrow \pi_1(M, x_\lambda)
 \xrightarrow{(\flat)} \pi_1(M,x_0),
\end{align}
where the first map is induced by the natural inclusion $\lambda
\hookrightarrow M$ and the second map $(\flat)$ is the change of 
the base point with respect to the path $l_\lambda$, that is, the map 
defined as $[c]\mapsto [l_\lambda cl_\lambda^{-1}]$. 
We also associate an element $\phi_\Sigma(a)$ of $\pi_1(M, x_0)$ defined as
$[l_\mu l_{\lambda, \mu}^{-1}l_\lambda^{-1}]$ 
to each edge $a$ of $\mathcal{Y}_\Sigma$ 
when $i(a)$ and $t(a)$ are denoted by $\sigma_\lambda$ and $\sigma_\mu$ 
respectively. 
Then the twisted commutativity
\begin{align} \label{eq:twist}
\phi_\Sigma(a)\phi_{\Sigma,\lambda}([c])\phi_\Sigma(a)^{-1}=\phi_{\Sigma,\mu} \circ\psi^\Sigma_a([c])
\end{align}
straightforwardly holds for each element $[c]$ of $G_\lambda^\Sigma=\pi_1(\lambda,
x_\lambda)$ by the construction of $\phi_\Sigma(a)$. Furthermore, 
for each composable pair $(a,b)$ of $E^{(2)}(\mathcal{Y}_\Sigma)$ satisfying 
$i(b)=\sigma_C$, $t(b)=i(a)=\sigma_S$ and $t(a)=\sigma_N$, 
one may readily verify the equation
$\phi_{\Sigma,N}(g_{a,b}^\Sigma)\phi_\Sigma(ab)=\phi_\Sigma(a)\phi_\Sigma(b)$ 
by direct calculation. 
Therefore $\phi_\Sigma=(\{ \phi_{\Sigma, \lambda} \}_{\lambda \in \Lambda}, \{
\phi_\Sigma(a)\}_{a \in E(\mathcal{Y}_\Sigma)})$ defines a morphism from 
$G(\mathcal{Y}_\Sigma)$ to $\pi_1(M, x_0)$.

Let $\sigma_0$ denote the unique vertex of
$\mathcal{Y}_{\Sigma}$ whose corresponding connected component
$N_0$ of $M(\Sigma)$ contains $x_0$. 
Then the  morphism $\phi_\Sigma\colon G(\mathcal{Y}_\Sigma) \rightarrow \pi_1(M,x_0)$ induces a homomorphism
$\phi_{\Sigma,*} \colon 
\pi_1(G(\mathcal{Y}_\Sigma), \sigma_0) \to \pi_1(M, x_0)$ on the
fundamental groups (refer to \cite[Chapter~III.$\mathcal{C}$ Proposition~3.6]{BH}).

\begin{prop} \label{prop:pres}
Let $\Sigma$ be an essential tribranched surface in $M$.
Then the homomorphism $\phi_{\Sigma, *} \colon
 \pi_1(G(\mathcal{Y}_\Sigma),\sigma_0)\rightarrow \pi_1(M,x_0)$ induced
 from the morphism $\phi_\Sigma \colon G(\mathcal{Y}_\Sigma) \rightarrow
 \pi_1(M,x_0)$ is surjective.
\end{prop}

\begin{proof}
Let $\mathop{\scalebox{1.5}{\raisebox{-0.2ex}{$\ast$}}}_{\lambda\in \Lambda}
 \pi_1(\lambda, x_\lambda)$ denote the free product of all local groups
 $\pi_1(\lambda, x_\lambda)$. The Seifert--van Kampen theorem then
 implies that the canonical homomorphisms $\pi_1(\lambda, x_\lambda)
 \rightarrow \pi_1(M,x_0)$ induce an surjection
 $\mathop{\scalebox{1.5}{\raisebox{-0.2ex}{$\ast$}}}_{\lambda\in \Lambda}
 \pi_1(\lambda, x_\lambda) \twoheadrightarrow \pi_1(M,x_0)$. On the
 other hand, there exists a natural quotient map $\mathop{\scalebox{1.5}{\raisebox{-0.2ex}{$\ast$}}}_{\lambda\in \Lambda}
 \pi_1(\lambda, x_\lambda) \rightarrow
 \pi_1(G(\mathcal{Y}_\Sigma), \sigma_0)$ by the construction of 
$\pi_1(G(\mathcal{Y}_\Sigma), \sigma_0)$. The homomorphism
 $\phi_{\Sigma,*}$ is compatible with these homomorphisms, 
and thus $\phi_{\Sigma,*}$ is also surjective.
\end{proof}

Due to the surjectivity of $\phi_{\Sigma,*}$, 
we may say that the $3$-manifold group 
$\pi_1(M,x_0)$ has a non-trivial presentation in terms of
$\pi_1(G(\mathcal{Y}_\Sigma), \sigma_0)$; in other words, 
$\pi_1(M,x_0)$ admits a {\em splitting} with respect to 
the fundamental group $\pi_1(G(\mathcal{Y}_\Sigma), \sigma_0)$ of 
the $2$-complex of groups $G(\mathcal{Y}_\Sigma)$. In the next
subsection we study when the induced homomorphism $\phi_{\Sigma,*}$ is
injective (and is thus an isomorphism). 

\subsection{Strongly essential tribranched surfaces} \label{ssc:SETBS}

In order to describe a condition for the induced homomorphism
$\phi_{\Sigma, *}$ to be injective, we here introduce the notion of {\em
strongly essential} tribranched surfaces.

\begin{defn}[Strongly essential tribranched surfaces]
Let $\Sigma$ be an essential tribranched surface contained in $M$. 
We say that $\Sigma$ is {\em strongly essential} 
if it satisfies the following additional condition besides the
 conditions (ETBS1), (ETBS2) and (ETBS3) in Definition~\ref{def:ETBS}: 
\begin{enumerate}[labelindent=.5em, labelwidth=4em, labelsep*=1em, leftmargin=!, label=(ETBS\arabic*)]
\setcounter{enumi}{3}
\item for each connected component $N$ of $M(\Sigma)$, the natural 
functorial homomorphism $\pi_1(N)\rightarrow \pi_1(M)$ is injective.
\end{enumerate}
\end{defn}

In the rest of this section 
we consider a strongly essential tribranched surface $\Sigma$ contained
in a $3$-manifold $M$.
Due to the condition (ETBS4) 
and the twisted commutativity (\ref{eq:twist}) 
one readily verifies that the morphism $\phi_\Sigma\colon
G(\mathcal{Y}_\Sigma) \rightarrow \pi_1(M,x_0)$ defined in the
previous subsection is {\em injective} on each
local group $G_\lambda^\Sigma$, and thus the $2$\nobreakdash-complex of groups
$G(\mathcal{Y}_\Sigma)$ is {\em developable} due to
Proposition~\ref{prop:development}.
Note that the surjectivity of $\phi_{\Sigma,*}$ (Proposition~\ref{prop:pres}) implies the {\em connectedness} of the development
 $D(\mathcal{Y}_\Sigma, \phi_\Sigma)$ of $G(\mathcal{Y}_\Sigma)$ with respect to the morphism $\phi_\Sigma$; see
 \cite[Chapter~III.$\mathcal{C}$ 3.14]{BH} for details.  
In the following we shall verify that the development
$D(\mathcal{Y}_\Sigma, \phi_\Sigma)$ 
is not only connected but also {\em simply connected} by reconstructing 
it in another {\em geometric} manner 
(compare to the construction of trees associated to hypersurfaces 
in \cite[Section~1.4]{Shalen}).

\medskip
\paragraph{\bfseries Geometric construction of a development}

Consider the universal cover $\widetilde{M}$ of $M$ and let
$\widetilde{\Sigma}$ denote the preimage of $\Sigma$ under the universal
covering map $p_{\widetilde{M}} \colon \widetilde{M}\rightarrow M$. 
Then one readily shows 
by using covering space theory 
that $\widetilde{\Sigma}$ is also a tribranched surface, and 
the preimage $C(\widetilde{\Sigma})$ of $C(\Sigma)$ 
under $p_{\widetilde{M}}$ coincides with the set 
of branched points of $\widetilde{\Sigma}$. 
Furthermore, for each 
connected component $\widetilde{C}$ of $C(\widetilde{\Sigma})$ 
in the preimage of a connected component $C$ of $C(\Sigma)$
 under  $p_{\widetilde{M}}$, 
there exists a unique tubular neighbourhood $h_{\widetilde{C}} \colon
\widetilde{C} \times D^2 \rightarrow \widetilde{M}$ of $\widetilde{C}$
in $\widetilde{M}$ satisfying 
$p_{\widetilde{M}}(h_{\widetilde{C}}(x,t))=h_C(p_{\widetilde{M}}(x),
t)$. We define $U_{\widetilde{C}}$ as an open subspace
$h_{\widetilde{C}}(\widetilde{C} \times \mathring{D}^2)$ and 
set $S(\widetilde{\Sigma})$ as $\widetilde{\Sigma} \setminus
\bigcup_{\widetilde{C} \in \pi_0(C(\widetilde{\Sigma}))}
U_{\widetilde{C}}$. 
Then, for each connected component $\widetilde{S}$ 
of $S(\widetilde{\Sigma})$ 
in the preimage of a connected component $S$ of $S(\Sigma)$ under $p_{\widetilde{M}}$, 
there exists a unique bicollar neighbourhood $h_{\widetilde{S}} \colon
\widetilde{S} \times [-1,1] \rightarrow \widetilde{M}\setminus
\bigcup_{\widetilde{C} \in \pi_0(C(\widetilde{\Sigma}))}
U_{\widetilde{C}}$ of $\widetilde{S}$ in $\widetilde{M}\setminus
\bigcup_{\widetilde{C} \in \pi_0(C(\widetilde{\Sigma}))}
U_{\widetilde{C}}$ satisfying
$p_{\widetilde{M}}(h_{\widetilde{S}}(x,t))=h_S(p_{\widetilde{M}}(x),
t)$. We define $U_{\widetilde{S}}$ as an 
open subspace $h_{\widetilde{S}}(\widetilde{S}\times(-1,1))$, 
and define $M(\widetilde{\Sigma})$ as 
the complement of $\bigcup_{\widetilde{C}\in
\pi_0(C(\widetilde{\Sigma}))} U_{\widetilde{C}} \cup
\bigcup_{\widetilde{S}\in
\pi_0(S(\widetilde{\Sigma}))}U_{\widetilde{S}}$ in $\widetilde{M}$.
We remark that $S(\widetilde{\Sigma})$ and $M(\widetilde{\Sigma})$
coincide with the preimages of $S(\Sigma)$ and $M(\Sigma)$
under $p_{\widetilde{M}}$ respectively. We now endow 
$\widetilde{M}$ with an equivalence relation 
$\sim_{\widetilde{\Sigma}}$ and construct a combinatorial CW-complex $Y_{\widetilde{\Sigma}}$
of dimension~$2$ as the quotient space
$Y_{\widetilde{\Sigma}}=\widetilde{M}/\sim_{\widetilde{\Sigma}}$, 
in the completely same manner as the construction of $Y_{\Sigma}$. 
By definition there exists a quotient map $r_{\widetilde{\Sigma}} \colon 
\widetilde{M}\rightarrow Y_{\widetilde{\Sigma}}$, 
and it is  easy to construct a continuous map $i_{\widetilde{\Sigma}} \colon
Y_{\widetilde{\Sigma}} \rightarrow \widetilde{M}$ such that 
$r_{\widetilde{\Sigma}} \circ i_{\widetilde{\Sigma}}$ 
is homotopic to the identity map on $Y_{\widetilde{\Sigma}}$. 
The composition of the induced maps 
\begin{align*}
\pi_1(Y_{\widetilde{\Sigma}}) \xrightarrow{i_{\widetilde{\Sigma},*}} \pi_1(\widetilde{M})
 \xrightarrow{r_{\widetilde{\Sigma},*}} \pi_1(Y_{\widetilde{\Sigma}})
\end{align*}
is thus the identity map. On the other hand the fundamental group
$\pi_1(\widetilde{M})$ of the universal cover $\widetilde{M}$ is trivial. 
Consequently $\pi_1(Y_{\widetilde{\Sigma}})$ is also trivial,
or in other words, $Y_{\widetilde{\Sigma}}$ is {\em simply
connected}. Note that the simply connected combinatorial
CW-complex $Y_{\widetilde{\Sigma}}$ admits an action of $\pi_1(M,x_0)$ 
induced from its natural action on $\widetilde{M}$. 
Moreover one readily checks by construction that 
the induced action of $\pi_1(M,x_0)$ on $Y_{\widetilde{\Sigma}}$ 
satisfies the following property; 
\begin{quotation}
$(\star)$ \quad an element $\gamma$ of $\pi_1(M,x_0)$  pointwisely fixes 
a cell $e_\lambda$ of dimension $1$ or $2$ if 
it stabilises $e_\lambda$.
\end{quotation}

Now let $\mathcal{Y}_{\widetilde{\Sigma}}$ denote the scwol associated to
$Y_{\widetilde{\Sigma}}$, which is constructed in the same manner as
$\mathcal{Y}_\Sigma$. Due to $(\star)$, the action of
$\pi_1(M,x_0)$ on $Y_{\widetilde{\Sigma}}$ induces its action 
on the scwol $\mathcal{Y}_{\widetilde{\Sigma}}$. Obviously by construction, the equalities $V(\mathcal{Y}_\Sigma)=V(\pi_1(M,x_0)\backslash \mathcal{Y}_{\widetilde{\Sigma}})$ and $E(\mathcal{Y}_\Sigma)=E(\pi_1(M,x_0)\backslash \mathcal{Y}_{\widetilde{\Sigma}})$ hold.

\begin{prop} \label{prop:sc}
The $2$-complex of groups $G(\mathcal{Y}_\Sigma)$ is isomorphic to
the complex of groups associated to the action of $\pi_1(M, x_0)$ 
on the scwol $\mathcal{Y}_{\widetilde{\Sigma}}$ constructed as above, and 
the morphism $\phi_\Sigma \colon G(\mathcal{Y}_{\Sigma}) \rightarrow
 \pi_1(M,x_0)$ coincides with the morphism associated to this action
 $($up to homotopy$)$. 
In particular, the scwol $\mathcal{Y}_{\widetilde{\Sigma}}$ is
 $\pi_1(M,x_0)$-equivariantly isomorphic to the development
 $D(\mathcal{Y}_\Sigma, \phi_\Sigma)$ of $G(\mathcal{Y}_\Sigma)$ with respect
 to $\phi_\Sigma$.
\end{prop}

\begin{proof} 
Recall that $p_{\widetilde{M}}\colon \widetilde{M}\rightarrow M$ denotes
 the universal cover of $M$. Take an arbitrary point $\tilde{x}_0$ from
 $p_{\widetilde{M}}^{-1}(x_0)$. 
For each $\lambda$ in $\Lambda$, let $\tilde{l}_\lambda$ denote 
a unique lift of $l_\lambda$ to $\widetilde{M}$ satisfying
 $\tilde{l}_\lambda(0)=\tilde{x}_0$. We set
 $\tilde{x}_\lambda=\tilde{l}_\lambda(1)$ and denote by
 $\tilde{\lambda}$ a unique connected component of
 $p_{\widetilde{M}}^{-1}(\lambda)$ containing $\tilde{x}_\lambda$. 
Note that $\tilde{x}_\lambda$ is a lift of $x_\lambda$ to
 $\widetilde{M}$.
We shall verify that all the data 
of which the complex of groups $G(\mathcal{Y}_\Sigma)$
 consists (specifically the local groups $G^\Sigma_\lambda$, the
 local homomorphisms $\psi^\Sigma_a$ and the twisting 
elements $g^\Sigma_{a,b}$) are obtained from the action of
 $\pi_1(M,x_0)$ on the scwol $\mathcal{Y}_{\widetilde{\Sigma}}$. 

Via the monodromy homomorphism and the parallel translation, we may identify
 $\pi_1(M,x_0)$  with the automorphism group of
 $p_{\widetilde{M}}^{-1}(\{ x_\lambda \})$. On the other hand, 
since the map (\ref{eq:identify}) is injective due to (ETBS2) and (ETBS4), 
one readily sees that the restriction $p_{\widetilde{M}}\lvert_{\tilde{\lambda}}\colon \tilde{\lambda}\rightarrow \lambda$ becomes the universal covering map.
The isotropy subgroup $\pi_1(M,x_0)_{\tilde{\lambda}}$ of $\pi_1(M,x_0)$ at
 $\tilde{\lambda}$ is then identified with the group of covering automorphisms of 
$p_{\widetilde{M}}\colon \widetilde{M}\rightarrow M$ stabilising $\tilde{\lambda}$, and the latter group coincides with the image of $\pi_1(\lambda,
 x_\lambda)$ in $\pi_1(M, x_0)$ under the injection
 (\ref{eq:identify}). We  may thus conclude that 
$G^\Sigma_\lambda=\pi_1(\lambda, x_\lambda)$ is the isotropy subgroup 
of $\pi_1(M,x_0)$ at $\tilde{\lambda}$ (or $\sigma_{\tilde{\lambda}}$)
 with respect to the natural action of $\pi_1(M,x_0)$ on $\widetilde{M}$
 (or on $\mathcal{Y}_{\widetilde{\Sigma}}$).

Next we verify that, for an appropriate choice of $h_a\in \pi_1(M,x_0)$ for each edge $a$ of the scwol $\mathcal{Y}_\Sigma=\pi_1(M,x_0)\backslash \mathcal{Y}_{\widetilde{\Sigma}}$, the equalities $\psi_a(\,-\,)=h_a\,(-)\,h_a^{-1}$ and $g_{a,b}=h_ah_bh_{ab}^{-1}$ hold. Let $a$ be an edge of $\mathcal{Y}_\Sigma$ and denote its initial
 and terminal vertices by $\sigma_\lambda$ and $\sigma_\mu$
 respectively. 
Let $\tilde{a}$ be a unique edge of $\mathcal{Y}_{\widetilde{\Sigma}}$ 
which is a lift of $a$ and satisfies $i(\tilde{a})=\sigma_{\tilde{\lambda}}$. 
We may identify $\tilde{a}$ with a unique lift 
$\tilde{l}_{\lambda, \mu}$ of $l_{\lambda, \mu}$ to $\widetilde{M}$ 
satisfying $\tilde{l}_{\lambda, \mu}(0)=\tilde{x}_\lambda$ up to homotopy. 
Then the parallel translation along the path $\tilde{l}_{\lambda,\mu}^{-1}\tilde{l}_{\lambda}^{-1}\tilde{l}_\mu$ defines an element $\tilde{h}_a$ of $\mathrm{Aut}(p^{-1}_{\widetilde{M}}(\{x_\mu\}))\cong \pi_1(M,x_\mu)$ satisfying $\tilde{h}_a.t(\tilde{a})=\widetilde{t(a)}=\sigma_{\tilde{\mu}}$. 
We denote by $h_a=[l_\mu l_{\lambda,\mu}^{-1}l_\lambda^{-1}]$ the image of $\tilde{h}_a$ in $\pi_1(M, x_0)$ under the change of base points $\pi_1(M, x_\mu)
 \xrightarrow{(\flat)} \pi_1(M, x_0)$ appearing in (\ref{eq:identify}). 
Then the image of 
 an element $\xi$ of $G^\Sigma_{\lambda}=\pi_1(\lambda, x_\lambda)$ 
 in $\pi_1(M,x_0)$ under the injection (\ref{eq:identify}) is none other than
 $[l_\lambda]\xi[l_\lambda^{-1}]$, and we may calculate as
\begin{align*}
h_a \xi h_a^{-1} = [l_\mu l_{\lambda,
 \mu}^{-1}l_\lambda^{-1}]\bigl([l_\lambda]\xi[l_\lambda^{-1}]\bigr)[l_\lambda
 l_{\lambda, \mu} l_\mu^{-1}] =
 [l_\mu l_{\lambda, \mu}^{-1}]\xi[l_{\lambda, \mu} l_\mu^{-1}],
\end{align*}
which is regarded as an element 
of $\pi_1(\mu,x_\mu)$ defined by $[l_{\lambda, \mu}^{-1}]\xi[l_{\lambda, \mu}]$ via the injection (\ref{eq:identify}). We thus obtain $\psi_a(\xi)=[l_{\lambda,\mu}^{-1}]\xi[l_{\lambda,\mu}]=h_a\xi h_a^{-1}$. Similarly 
we may calculate as 
\begin{align*}
h_a h_b h_{ab}^{-1}
 &=[l_N l_{S, N}^{-1} l_S^{-1}][l_S l_{C, S}^{-1} l_C^{-1}][l_C l_{C, N} l_N^{-1}] \\
&= [l_N l_{S, N}^{-1} l_{C, S}^{-1} l_{C, N} l_N^{-1}] \qquad (\text{as an element
 of $\pi_1(M,x_0)$}) \\
&= [l_{S, N}^{-1} l_{C, S}^{-1} l_{C, N}]=g_{a,b}^\Sigma \qquad (\text{as an
 element of $\pi_1(N,x_N)$})
\end{align*}
for composable edges $a$ and $b$. Here $C$, $S$ and $N$ denote elements
 of $\pi_0(C(\Sigma))$, $\pi_0(S(\Sigma))$ and $\pi_0(M(\Sigma))$
 respectively such that $i(b)=\sigma_C$, $t(b)=i(a)=\sigma_S$
 and $t(a)=\sigma_N$ hold.  
Moreover the equality
\begin{align*}
\phi_{\Sigma}(a) =[l_\mu l_{\lambda, \mu}^{-1}l_\lambda^{-1}]=h_a
\end{align*}
obviously holds for an edge $a$ with $i(a)=\sigma_\lambda$ and
 $t(a)=\sigma_\mu$. Therefore, under the specific choices of a lift
 $\tilde{\sigma}$ of each vertex $\sigma$ of $\mathcal{Y}_{\Sigma}$ and 
an element $h_a$ of $\pi_1(M,x_0)$ for each edge $a$ of
 $\mathcal{Y}_\Sigma$ as 
\begin{align*}
\tilde{\sigma}_\lambda &=\sigma_{\tilde{\lambda}} \quad (\lambda \in
 \Lambda), & h_a &= [l_\mu l_{\lambda,\mu}^{-1}l_\lambda^{-1}] \quad
 \text{for an edge $a$ with }i(a)=\sigma_\lambda,\ t(a)=\sigma_\mu,
\end{align*}
the complex of groups $G(\mathcal{Y}_\Sigma)$ is indeed 
the one associated to the action of $\pi_1(M,x_0)$ on the
 scwol $\mathcal{Y}_{\widetilde{\Sigma}}$, and $\phi_\Sigma$ is the
 associated morphism.

The rest of the statement is then a direct consequence of
 \cite[Chapter~III.$\mathcal{C}$ Theorem~2.13 (2)]{BH}.
\end{proof}

Note that the combinatorial CW-complex $Y_{\widetilde{\Sigma}}$ of dimension~$2$ is 
regarded as the geometric realisation of the scwol
$\mathcal{Y}_{\widetilde{\Sigma}}$.
Since the geometric realisation $\lvert
\mathcal{Y}_{\widetilde{\Sigma}}\rvert=Y_{\widetilde{\Sigma}}$ 
of $\mathcal{Y}_{\widetilde{\Sigma}}$
is simply-connected as we have observed, so is 
$\mathcal{Y}_{\widetilde{\Sigma}}$ itself due to
Proposition~\ref{prop:fundamental}. 
Consequently the
scwol $\mathcal{Y}_{\widetilde{\Sigma}}$ is {\em connected} and {\em simply
connected}, and Proposition~\ref{prop:sc} implies that the development 
$D(\mathcal{Y}_\Sigma, \phi_\Sigma)$ of $G(\mathcal{Y}_\Sigma)$ 
with respect to $\phi_\Sigma$ is also connected and simply
connected. 
Then by basic facts
of covering space theory on complexes of groups (see
\cite[Chapter~III.$\mathcal{C}$ 3.14 (2)]{BH} for details), 
we obtain the following result.

\begin{thm}
Let $\Sigma$ be a strongly essential tribranched surface contained in 
a compact, connected, irreducible and orientable $3$-manifold $M$. 
Then the morphism $\phi_\Sigma \colon G(\mathcal{Y}_\Sigma) \rightarrow
 \pi_1(M,x_0)$ constructed in Section~$\ref{ssc:cgtbs}$ induces a
 group isomorphism $\phi_{\Sigma, *} \colon \pi_1(G(\mathcal{Y}_\Sigma),
 \sigma_0) \xrightarrow{\sim} \pi_1(M,x_0)$.
\end{thm}

\section{The Bruhat--Tits buildings $\mathcal{B}(\mathrm{SL}(n)_{/F})$ associated to the special linear groups} \label{sc:BT}

{\em Bruhat--Tits buildings} are combinatorial and topological objects
associated to reductive algebraic groups
defined over non-archimedean valuated fields,
which behave as Riemannian symmetric spaces in differential
geometry; in particular they admit natural ``transitive'' actions 
of the algebraic groups (to be precise, the natural group actions on the
Bruhat--Tits buildings are {\em strictly transitive}; 
see the end of Section~\ref{ssc:EB} for the definition of 
strict transitivity). 
The theory of Bruhat--Tits buildings has its origin in the study of 
Nagayoshi Iwahori and Hideya Matsumoto on the generalised Bruhat 
decomposition of $\mathfrak{p}$-adic Chevalley groups \cite{IM}, 
and then it has been elaborated by Fran\c{c}ois Bruhat and
Jacques Tits in a systematic and axiomatic way \cite{BT1,BT2}. 
The Bruhat--Tits tree, which appears 
in the work of Culler and Shalen \cite{CS}, is none other 
than the Bruhat--Tits building associated to 
the special linear group $\mathrm{SL}(2)$ of degree~$2$, and 
the Bruhat--Tits buildings associated 
to the special linear groups {\em of higher degree} play 
crucial roles in our extension of Culler and Shalen's results.
In this section we shall summarise basic notion 
on Bruhat--Tits buildings and their fundamental properties especially for 
the special linear groups. 

\subsection{Euclidean buildings and their contractibility} \label{ssc:EB}

We first review the axiomatic definition of (Euclidean) buildings 
after Tits and basic properties of Euclidean buildings. 
Refer, for instance, to \cite{AB, Garrett} for details 
of the contents of this subsection.

\begin{defn}[Chamber complexes]
Let $\Sigma$ be an abstract simplicial complex of finite dimension (that
 is, every simplex of $\Sigma$ is of finite dimension). We call $\Sigma$ 
a {\em chamber complex} if the following two conditions are fulfilled:
\begin{enumerate}[label=(CC\arabic*)]
\item every maximal simplex of $\Sigma$ has the same dimension $n$;
\item every two maximal simplices $C$ and $C'$ are connected by a {\em
      gallery}; that is, there exists a sequence 
      of maximal simplices $C_0=C,\ C_1,\dotsc, C_r=C'$ of $\Sigma$ such that 
      $C_{i-1}$ and $C_i$ are adjacent for each $1\leq i\leq r$.
\end{enumerate}
\end{defn}

Here we say that maximal simplices $C$ and $C'$ of $\Sigma$ are {\em
adjacent} if $C$ and $C'$ are distinct and contain a common
$(n-1)$\nobreakdash-dimensional face. A maximal simplex of $\Sigma$ is
called a {\em chamber} of $\Sigma$.
The {\em dimension} of $\Sigma$ is defined as the (same) dimension $n$ 
of a chamber of $\Sigma$.
A chamber complex $\Sigma$ of dimension~$n$ 
is said to be {\em thin} if every $(n-1)$\nobreakdash-dimensional
simplex of $\Sigma$ is a face of exactly two chambers. 

\begin{defn}[Buildings]
Let $\Delta$ be an abstract simplicial complex. We call $\Delta$ a {\em
$($simplicial, thick$)$ building} of dimension $n$
if there exists a family $\mathcal{A}$ of $n$\nobreakdash-dimensional 
 thin chamber subcomplexes of $\Delta$ and the pair $(\Delta,
 \mathcal{A})$ satisfies the following axioms:
\begin{enumerate}[label=(B\arabic*)]
\setcounter{enumi}{-1}
\item the complex $\Delta$ is (set-theoretically) expressed as the union
      of all elements of $\mathcal{A}$,
      and each $(n-1)$\nobreakdash-dimensional simplex of $\Delta$ is a
      face of at least three maximal simplices of $\Delta$ 
      (which are of dimension~$n$);
\item every two simplices of $\Delta$ 
      are contained in a single chamber subcomplex of
      $\Delta$ belonging to $\mathcal{A}$;
\item if $\Sigma$ and $\Sigma'$ are elements of $\mathcal{A}$ both of
      which contain two simplices $\sigma$ and $\tau$, there
      exists an isomorphism $\Sigma\xrightarrow {\sim} \Sigma'$ of 
      chamber complexes which fixes all the vertices of $\sigma$ and $\tau$.
\end{enumerate}
\end{defn}

A thin chamber subcomplex $\Sigma$ of $\Delta$ belonging to $\mathcal{A}$ is
called an {\em apartment} of $\Delta$, and a maximal simplex of $\Delta$
is called a {\em chamber} of $\Delta$. 
Among families of thin chamber subcomplexes of $\Delta$ satisfying 
all the axioms (B0), (B1) and (B2), there exists a unique maximal one
$\mathcal{A}^\mathrm{cpl}$ which is called 
the {\em complete system of apartments} of $\Delta$. 

It is well known that a building $\Delta$ of dimension~$n$ 
is a {\em colorable} chamber complex; 
namely there exists a (set-theoretical) $I_{n+1}$-valued function $\tau$
on the vertices of $\Delta$ such that the vertices of each chamber of
$\Delta$ are mapped bijectively onto $I_{n+1}$, 
where $I_{n+1}$ denotes a finite set of cardinality~$n+1$. 
Such a function $\tau$ is called a {\em type function on $\Delta$} 
(with values in $I_{n+1}$). We refer the reader to
\cite[Proposition~4.6]{AB} for details.

\begin{defn}[Euclidean buildings]
A building $\Delta$ of dimension~$n$ is 
said to be a {\em Euclidean building} $($or a {\em building of affine
 type}$)$ if the geometric realisation of each apartment of $\Delta$ is 
isomorphic to the standard tessellation of 
the $n$\nobreakdash-dimensional (real) Euclidean space 
by equilateral $n$\nobreakdash-dimensional 
simplices (more precisely, we require that 
each apartment should be isomorphic to a Euclidean Coxeter complex).
\end{defn}

Now let $\Delta$ be a Euclidean building. 
For arbitrary two points $x$ and $y$ of the geometric realisation $|\Delta|$
of $\Delta$, there exists an apartment $\Sigma_{(x,y)}$ of $\Delta$ whose
geometric realisation $|\Sigma_{(x,y)}|$ contains both of $x$ and $y$ due to the axiom (B1) of
buildings. We equip $|\Sigma_{(x,y)}|$  
with the standard Euclidean metric $d_{|\Sigma_{(x,y)}|}$, 
and define a real-valued function $d_{|\Delta|}$ on $\lvert \Delta
\rvert \times \lvert \Delta \rvert$ by
\begin{align*}
d_{|\Delta|} \colon |\Delta| \times |\Delta| \rightarrow
 \mathbb{R}_{\geq 0} \ ; \ (x,y) \mapsto d_{|\Sigma_{(x,y)}|}(x,y).
\end{align*}
Then $d_{|\Delta|}$ is a metric on the geometric
realisation $|\Delta|$ of $\Delta$ which is well defined independently 
of the choice of an apartment $\Sigma_{(x,y)}$ 
due to the axiom (B2) of buildings.
One readily checks that the topology of $|\Delta|$ determined by the metric
$d_{|\Delta|}$ coincides with the weak topology endowed on $|\Delta|$. 
Bruhat and Tits have verified that 
the metric space $(|\Delta|, d_{|\Delta|})$ is 
a CAT$(0)$~space; in particular $|\Delta|$ is {\em contractible}
(refer to \cite[Propositions 2.5.3.\ et 2.5.16]{BT1} for details; see
also \cite[the proof of Theorem~11.16]{AB}). 
The contractibility of Euclidean buildings shall play a crucial role 
in the construction of tribranched surfaces in Section~\ref{ssc:NAB}. 

We shall end this subsection by presenting several notion
concerning group actions on buildings. Let $G$ be an abstract group and 
$\Delta$ a building on which $G$ acts.
One easily verifies that the action of $G$ on
$\Delta$ induces actions of $G$ both on the complete system of apartments
$\mathcal{A}^\mathrm{cpl}$ of $\Delta$ and on the set of 
all the chambers of $\Delta$. 
An action of a group $G$ on a building $\Delta$ is said to be 
{\em strictly transitive} if $G$ acts transitively  
on the set of all pairs $(\Sigma, C)$ consisting of 
an apartment $\Sigma$ (belonging to $\mathcal{A}^\mathrm{cpl}$) and 
a chamber $C$ contained in $\Sigma$, and said to be {\em
type-preserving} if an arbitrary element $\gamma$ of $G$ maps a vertex of 
$\Delta$ to one of the same type (with respect to a certain type
function on $\Delta$).

\subsection{Combinatorial construction of $\mathcal{B}(\mathrm{SL}(n)_{/F})$} \label{ssc:BT1}

One of the most significant aspects in the theory of Euclidean buildings 
is the fact that one may associate in a canonical manner a Euclidean building
$\mathcal{B}(G_{/F})$ to a reductive algebraic group $G$ defined 
over a non-archimedean valuated field $F$. Furthermore $\mathcal{B}(G_{/F})$
admits a natural, strictly transitive action of $G(F)$.   
The existence of such Euclidean buildings was first observed 
in the pioneering work of Iwahori and Matsumoto \cite{IM} for Chevalley
groups (which are in particular split, semisimple and 
simply connected algebraic groups) defined 
over $\mathfrak{p}$\nobreakdash-adic fields.\footnote{More precisely,
Iwahori and Matsumoto have constructed a (generalised) BN pair
with respect to the {\em Iwahori subgroup} $B$ of a
$\mathfrak{p}$-adic Chevalley group in \cite[Proposition~2.2,
Theorem~2.22]{IM}. Although they have never mentioned buildings in
\cite{IM}, it is well known that one may associate buildings to 
such BN-pairs in a canonical way; see \cite[Theorem~6.56]{AB} for example.} 
Then Bruhat and Tits established
construction of such Euclidean buildings in \cite{BT1,BT2} 
for general reductive algebraic groups. The Euclidean building 
$\mathcal{B}(G_{/F})$ attached to $G_{/F}$ is therefore 
called the {\em Bruhat--Tits building} associated to $G_{/F}$. 

Bruhat and Tits's construction of $\mathcal{B}(G_{/F})$ 
utilising ``valuated root data'' is rather abstract and complicated, 
but limiting ourselves to the Bruhat--Tits building $\mathcal{B}(G_{/F})$ 
associated to the {\em special linear group} $G=\mathrm{SL}(n)$ 
defined over a discrete valuation field 
(which is a $\mathfrak{p}$-adic Chevalley group and thus has been already
 dealt with by Iwahori and Matsumoto in \cite{IM}),  
we may explicitly describe the combinatorial structure of
$\mathcal{B}(G_{/F})$ and the effect of the action of $G(F)$ on $\mathcal{B}(G_{/F})$ 
without introducing any root datum. 
We propose in this subsection a combinatorial description of 
the Bruhat--Tits building $\mathcal{B}(\mathrm{SL}(n)_{/F})$ 
associated to the special linear group $\mathrm{SL}(n)_{/F}$, 
mainly following \cite[Chapter~19]{Garrett}.
We shall only utilise the Bruhat--Tits buildings $\mathcal{B}(\mathrm{SL}(n)_{/F})$ 
associated to the special linear groups in our later applications.

\medskip
Let $F$ be a field equipped with a (normalised) 
discrete valuation $w \colon F^\times \rightarrow
\mathbb{Z}$. We do not require that 
the base field $F$ is complete with respect to 
the multiplicative valuation $\lvert\cdot\rvert_w$ 
associated to $w$ (indeed we shall later apply results 
of this subsection to a case where the base field is not complete). 
We denote the valuation ring of $F$ with respect to $w$ by $\mathcal{O}_w$. 
We fix a uniformiser $\varpi_w$ 
of the discrete valuation field $(F, w)$; in other words, 
we choose and fix a generator $\varpi_w$ of 
the maximal ideal of $\mathcal{O}_w$ (which is known 
to be a principal ideal due to basic facts of valuation theory).   

Let $V_n$ denote an $n$\nobreakdash-dimensional vector space over $F$ 
equipped with a basis $\{e_1, \dotsc, e_n \}$. We identify 
$V_n$ with $F^{\oplus n}$ (the $F$-vector space of
$n$\nobreakdash-dimensional column vectors) with respect to 
the specified basis $\{e_j \}_{j=1}^n$ and regard
the special linear group $\mathrm{SL}_n(F)$ 
as a subgroup of $\mathrm{Aut}_F(V_n)$. An
$\mathcal{O}_w$\nobreakdash-submodule $L$ of $V_n$ is 
called a {\em lattice} of $V_n$
if $L$ spans $V_n$ over $F$: $\langle L\rangle_F=V_n$. 
Every lattice of $V_n$ is then a free $\mathcal{O}_w$\nobreakdash-module 
of rank~$n$ by elementary divisor theory. Two lattices $L$ and $L'$ of
$V_n$ are 
said to be {\em homothetic} if there exists a nonzero element $a$ of
$F$ such that $L$ coincides with $aL'$ (as 
an $\mathcal{O}_w$\nobreakdash-submodule of $V_n$). 
The homothety relation is an equivalence relation on the set of all
lattices of $V_n$, and we define the vertex set 
$V(\mathcal{B}(\mathrm{SL}(n)_{/(F,w)}))$ of the Bruhat--Tits building
$\mathcal{B}(\mathrm{SL}(n)_{/(F,w)})$ as the set of homothety classes of lattices of $V_n$. 
We say that two distinct elements $v$ and $v'$ of
$V(\mathcal{B}(\mathrm{SL}(n)_{/(F,w)}))$ are {\em adjacent} 
if there exist lattices $L$ and $L'$ representing
the homothety classes $v$ and $v'$ respectively 
such that
\begin{align*}
\varpi_wL' \subsetneq L \subsetneq L'
\end{align*}
holds (as $\mathcal{O}_w$-submodules of $V_n$). 
We then define $\mathcal{B}(\mathrm{SL}(n)_{/(F,w)})$ as an abstract simplicial complex 
each of whose simplices is a finite subset $\{v_1, \dotsc, v_r\}$ of 
$V(\mathcal{B}(\mathrm{SL}(n)_{/(F,w)}))$
consisting of vertices adjacent to each other; in other
words, a set $\{v_1, \dotsc, v_r \}$ of $r$~vertices of
$\mathcal{B}(\mathrm{SL}(n)_{/(F,w)})$ forms an $r$\nobreakdash-simplex 
if and only if there exists a lattice $L_i$ 
representing $v_i$ for each $1\leq i\leq r$ such that
\begin{align*}
\varpi_wL_r \subsetneq L_1 \subsetneq L_2 \subsetneq \cdots \subsetneq L_r
\end{align*}
holds (after appropriate relabeling of the subindices). 
For an arbitrary $F$\nobreakdash-basis 
$\mathbf{f}=\{f_1, \dotsc, f_n\}$ of $V_n$, 
consider a subcomplex $\Sigma_\mathbf{f}$ of $\mathcal{B}(\mathrm{SL}(n)_{/(F,w)})$ 
generated by the homothety classes of lattices 
of the form $\sum_{j=1}^n \mathcal{O}_w\varpi_w^{m_j}f_j$ 
(each $m_j$ takes an arbitrary integer). The subcomplex $\Sigma_\mathbf{f}$ is 
indeed a thin chamber complex of dimension~$n-1$.
Denote by $\mathcal{A}$ 
the family of the subcomplexes $\Sigma_\mathbf{f}$ of $\mathcal{B}(\mathrm{SL}(n)_{/(F,w)})$ 
indexed by an $F$\nobreakdash-basis $\mathbf{f}$ of $V_n$. 
Then we may readily verify that the pair $(\mathcal{B}(\mathrm{SL}(n)_{/(F,w)}), \mathcal{A})$ 
satisfies all the axioms (B0), (B1) and (B2) of buildings; 
see \cite[Chapter~19.2]{Garrett} for details. 
The special linear group $\mathrm{SL}_n(F)$ 
acts on the set of lattices of $V_n$ in an obvious manner; 
namely, for a lattice $L=\sum_{j=1}^n \mathcal{O}_w
f_j$ with an $\mathcal{O}_w$\nobreakdash-basis $\{f_1, \dotsc, f_n \}$, 
we define $gL$ as an $\mathcal{O}_w$\nobreakdash-submodule of $V_n$ 
spanned by $\{g(f_1), \dotsc, g(f_n)\}$ (here we regard
$g$ as an element of $\mathrm{Aut}_F(V_n)$). This defines 
an action of $\mathrm{SL}_n(F)$ 
on $V(\mathcal{B}(\mathrm{SL}(n)_{/(F,w)}))$, 
which is naturally extended to an action of $\mathrm{SL}_n(F)$ on $\mathcal{B}(\mathrm{SL}(n)_{/(F,w)})$.
One of the significant features of the action of $\mathrm{SL}_n(F)$ on
$\mathcal{B}(\mathrm{SL}(n)_{/(F,w)})$ is that 
it is a strictly transitive and type-preserving action. 
In particular, an element $\gamma$ of $\mathrm{SL}_n(F)$ fixes {\em all the vertices} 
of a chamber $C$ whenever $\gamma$ stabilises $C$. 

In order to see that it is type-preserving, one has only to 
check that an association of a value $\tau(v)=(w(\det g_v)\!\!\mod n)$ 
to each vertex $v$ of
$\mathcal{B}(\mathrm{SL}(n)_{/(F,w)})$ defines a type function $\tau$ on
$\mathcal{B}(\mathrm{SL}(n)_{/(F,w)})$ with values in $\mathbb{Z}/n\mathbb{Z}$.
Here $g_v$ is an element of $\mathrm{Aut}_F(V_n)$  satisfying
$L=g_v(L_0)$ for a certain lattice $L$ representing $v$, and $L_0$
denotes the standard lattice of $V_n$
defined as $L_0=\sum_{j=1}^n \mathcal{O}_we_j$.
Then the type of a vertex of
$\mathcal{B}(\mathrm{SL}(n)_{/(F,w)})$ does not change under the action of an
element $\gamma$ of $\mathrm{SL}_n(F)$ since one has
\begin{align*}
\tau(\gamma v)=(w(\det (\gamma g_v))\!\! \mod n)=(w(\det \gamma)\!\!
 \mod n)+ \tau(v)=\tau(v)
\end{align*}
by using $\det(\gamma)=1$.

\begin{rem}
The Bruhat--Tits building  $\mathcal{B}(\mathrm{GL}(n)_{/(F,w)})$ 
associated to the {\em general linear group}
 $\mathrm{GL}(n)_{/(F,w)}$ is completely the same one as
 $\mathcal{B}(\mathrm{SL}(n)_{/(F,w)})$. However, the
 natural action of $\mathrm{GL}_n(F)$ on $\mathcal{B}(\mathrm{GL}(n)_{/(F,w)})$ does
 {\em not} preserve the type function $\tau(v)=(w(\det g_v)\!\! \mod n)$
introduced above since the $\mathbb{Z}$-valued function $w\circ \det$ on $\mathrm{GL}_n(F)$ takes
 arbitrary value (indeed $\mathrm{GL}_n(F)$ acts transitively on the vertex set
 $V(\mathcal{B}(\mathrm{GL}(n)_{/(F,w)}))$). 
In order to guarantee that the natural action on the
 Bruhat--Tits building is type-preserving, we deal with the Bruhat--Tits
 building associated to the special linear group $\mathrm{SL}(n)$ rather than
 the Bruhat--Tits building associated to the general linear group
 $\mathrm{GL}(n)$. We shall effectively utilise 
the type-preserving property of the action when we consider 
the quotient complex 
$\mathcal{B}_{n,\widetilde{D},\tilde{y}}/\pi_1(M,x_0)$ in Section~\ref{ssc:construction}.
\end{rem}

\begin{exmp}[Bruhat--Tits trees]
In the case where $n$ equals $2$, the construction of $\mathcal{B}(\mathrm{SL}(2)_{/(F,w)})$ explained above is 
none other than the classical construction of 
the {\em Bruhat--Tits tree} associated to $\mathrm{SL}(2)_{/F}$, 
which is, for example, presented in \cite[Chapitre~II, Section~1]{Serre}. 
Note that the Bruhat--Tits trees play crucial roles in the original 
work of Culler and Shalen \cite{CS}.
\end{exmp}

\section{Construction of essential tribranched surfaces} \label{sc:construction}

We shall establish our construction of essential
tribranched surfaces in this section. There are two technical hearts 
in the construction. One is to obtain a {\em nontrivial} 
type-preserving action of the $3$-manifold group on the 
Bruhat--Tits building associated to the special linear group $\mathrm{SL}(n)$ 
by utilising geometry of character varieties of higher degree. 
After a brief review on character varieties of higher degree in
Section~\ref{ssc:charvar}, we explain how to obtain
such a nontrivial action in
Section~\ref{ssc:NAB}. The other is to construct a non-empty 
tribranched surfaces from such a nontrivial action.  
In Section~\ref{ssc:construction}, we put this procedure in practice,
and then modify the obtained tribranched surfaces to be {\em essential} by certain 
local surgery.

\subsection{$\mathrm{SL}_n(\mathbb{C})$-character variety} \label{ssc:charvar}

We begin with briefly reviewing 
the {\em $\mathrm{SL}_n(\C)$-character variety} of a finitely generated group.
See Lubotzky and Magid~\cite{LM} for more details.

Let $\pi$ be a finitely generated group.
We denote by $R_n(\pi)$ the set $\hom(\pi, \mathrm{SL}_n(\C))$ of all the
$\mathrm{SL}_n(\C)$\nobreakdash-representations of $\pi$, which is an affine algebraic set.
The algebraic group $\mathrm{SL}(n)_{/\C}$ acts on $R_n(\pi)$ by conjugation.
We denote by $X_n(\pi)$ the geometric invariant theoretical 
quotient of $R_n(\pi)$ with respect to this action, which is
called the {\em $\mathrm{SL}_n(\C)$-character variety} of $\pi$.
We define the {\em character}
$\chi_\rho \colon \pi \to \C$ of a $\mathrm{SL}_n(\C)$-representation $\rho
\colon \pi\to \mathrm{SL}_n(\C)$ as
$\chi_\rho(\gamma) = \tr \rho(\gamma)$ for each element $\gamma$ in $\pi$.
The quotient variety $X_n(\pi)$ is known to be realised as the set of characters
$\chi_\rho$ (in the set-theoretical sense), 
and under this identification the quotient map
$R_n(\pi) \to X_n(\pi)$ is regarded 
as the map which sends $\rho$ to $\chi_\rho$.
For an element $\gamma$ of $\pi$, 
we define the {\em trace function} $I_\gamma \colon X_n(\pi) \to \C$
associated to $\gamma$
as $I_\gamma(\chi_\rho) = \tr \rho(\gamma)$, which is a regular function
on $X_n(\pi)$.

The following theorem is a direct consequence of the result of
Claudio Procesi~\cite{Procesi}.
\begin{thm}[Procesi, {\cite[Theorem~3.4 (a)]{Procesi}}] \label{thm:Procesi}
Let $\gamma_1, \dots, \gamma_m$ be a generator system of $\pi$.
Then the trace functions 
$\{ I_{\gamma_{i_1} \dots \gamma_{i_k}} \}_{1 \leq i_1, \dots, i_k \leq m}^{1 \leq k \leq 2^n - 1}$
give affine coordinates of $X_n(\pi)$.
\end{thm}

For a compact $3$-manifold $M$ we abbreviate $X_n(\pi_1(M))$ as $X_n(M)$ 
to simplify notation.

\begin{rem}
Let $M$ be a hyperbolic $3$-manifold with $l$ torus cusps.
Then we may consider a lift $\rho_0 \colon \pi_1(M) \to \mathrm{SL}_2(\C)$ of 
the holonomy representation with respect to the hyperbolic structure of
 $M$~\cite[Proposition~3.1.1]{CS}.
Menal-Ferrer and Porti~\cite{MFP1,MFP2} showed for general $n$ 
the following facts;
\begin{enumerate}[label=\roman*), widest=ii]
\item the character variety $X_n(M)$ is smooth at $\chi_{\iota_n \circ \rho_0}$;
\item the irreducible component of $X_n(M)$ containing
$\chi_{\iota_n \circ \rho_0}$ is of dimension~$l(n-1)$.
\end{enumerate}
Here $\iota_n \colon \mathrm{SL}_2(\C) \to \mathrm{SL}_n(\C)$ denotes
an (arbitrary) irreducible representation.
They also gave explicit local coordinates around
$\chi_{\iota_n \circ \rho_0}$~\cite{MFP2}. 
When $n$ equals $2$, these results had been already proved by Kapovich \cite{Kapovich} (see
 also Bromberg \cite{Bromberg}).
\end{rem}

\begin{defn}[Ideal points] \label{def:ideal}
Suppose that $X_n(\pi)$ is of positive dimension and 
let us take an affine curve $C$ contained in $X_n(\pi)$. 
Let $\widetilde{C}\rightarrow C$ denote a desingularisation of 
a projective completion of $C$,
so that $\widetilde{C}$ is a smooth projective model of $C$.
A closed point $\tilde{x}$ of $\widetilde{C}$ is
 called an {\em ideal point of $C$} if the birational map
 $\widetilde{C}\rightarrow C$ above is undefined at $\tilde{x}$.
\end{defn}

Note that the notion of ideal points does not depend on the choices of
projective completions and desingularisations in the definition (see
\cite[Section~1.3]{CS} for details). 
We also remark that 
there are only finitely many ideal points  of $C$ on $\widetilde{C}$.

\subsection{Nontrivial actions on Bruhat--Tits buildings} \label{ssc:NAB}

We discuss in this subsection how to obtain a nontrivial,
type-preserving action of a finitely generated group $\pi$ 
on a Euclidean building. 
Such a nontrivial action gives rise to a nontrivial splitting 
of $\pi$, which shall play a central role in the construction 
of tribranched surfaces when $\pi$ is a $3$\nobreakdash-manifold group. 
Similarly to the arguments in \cite[Section~2.2]{CS}, we utilise geometry 
of the character variety associated to $\pi$ in order to obtain 
such an action.

Assume that the character variety $X_n(\pi)$ is of positive dimension
and consider an affine curve $C$ in $X_n(\pi)$. 
Then we may take a lift $D$ of $C$ in $R_n(\pi)$. Namely $D$
is an affine curve contained in the inverse image of $C$ under the natural 
projection $\mathrm{pr}_n \colon R_n(\pi) \rightarrow X_n(\pi)$ 
such that the restriction $\mathrm{pr}_n\vert_D$ is not a constant morphism. 
The projection $\mathrm{pr}_n\vert_D \colon D\rightarrow C$ induces 
a (surjective) regular morphism $\mathrm{pr}_n\vert_{D}^{\sim} \colon \widetilde{D}
\rightarrow \widetilde{C}$ on the smooth projective models of $C$ and $D$,
which sends the ideal points of $\widetilde{D}$ to those of
$\widetilde{C}$. 

Recall that, by the definition of $R_n(\pi)$, 
each closed point $y$ of the affine algebraic set $R_n(\pi)$ corresponds to an
$\mathrm{SL}_n(\C)$-representation $\rho_y\colon \pi \to \mathrm{SL}_n(\C)$. We denote by $\C[R_n(\pi)]$ the affine
coordinate ring of $R_n(\pi)$.
Let $\rho_\mathrm{taut} \colon \pi \rightarrow \mathrm{SL}_n(\C[R_n(\pi)])$ denote 
the tautological representation of $\pi$; 
namely $\rho_\mathrm{taut}(\gamma)$ is a regular
$\mathrm{SL}_n(\mathbb{C})$\nobreakdash-valued function on $R_n(\pi)$ for each
element $\gamma$ of $\pi$ whose value at a closed point $y$ of
$R_n(\pi)$ is $\rho_y(\gamma)$. 
Let $\rho_{\widetilde{D}} \colon \pi \rightarrow
\mathrm{SL}_n(\C(D))$ denote the composition of the tautological representation
$\rho_\mathrm{taut} \colon \pi \rightarrow \mathrm{SL}_n(\C[R_n(\pi)])$
with 
\begin{align*}
\mathrm{SL}_n(\C[R_n(\pi)])\rightarrow \mathrm{SL}_n(\C[D]) \hookrightarrow \mathrm{SL}_n(\C(D)),
\end{align*}
where the first map is induced by the natural embedding $D\hookrightarrow
R_n(\pi)$. In the construction of $\rho_{\widetilde{D}}$, we identify
$\C(D)$ with the field of rational functions of $\widetilde{D}$ 
due to the fact that $\widetilde{D}$ is birational to $D$ (this 
gives justification to the notation $\rho_{\widetilde{D}}$).  
We  call $\rho_{\widetilde{D}}$ the {\em
tautological representation associated to the affine curve $D$}. 
Now recall that a closed point $y$ of the smooth projective curve 
$\widetilde{D}$ (possibly an ideal point of $D$) determines 
a discrete valuation $w_y \colon \C(D)^\times
\rightarrow \mathbb{Z}; f \mapsto \mathrm{ord}_y(f)$ 
on the field of rational functions $\C(D)$ of $\widetilde{D}$ 
(that is, the {\em order function} at $y$; see \cite[Definition~(1.32)]{Mumford} for details). 
The {\em Bruhat--Tits building associated to
$(\widetilde{D},y)$} is then defined as
$\mathcal{B}_{n,\widetilde{D},y}=\mathcal{B}(\mathrm{SL}(n)_{/(\C(D), w_y)})$, 
which admits a canonical action of $\mathrm{SL}_n(\C(D))$. We thus obtain 
an action of $\pi$ on the Bruhat--Tits building $\mathcal{B}_{n,\widetilde{D},y}$
\begin{align*}
\pi \xrightarrow{\rho_{\widetilde{D}}} \mathrm{SL}_n(\C(D))
 \xrightarrow{\text{canonical}} \mathrm{Aut}(\mathcal{B}_{n,\widetilde{D},y})
\end{align*}
which is automatically type-preserving as we have already remarked in
Section~\ref{ssc:BT1}.

The following theorem is an analogue of Culler and Shalen's
``Fundamental Theorem'' \cite[Theorem~2.2.1]{CS} for 
representations of $\pi$ of higher dimension.

\begin{thm} \label{thm:FT}
Let $pr_n\vert_D^{\sim} \colon \widetilde{D}\rightarrow \widetilde{C}$
 be as above and let $y$ be a closed point of $\widetilde{D}$. 
Set $x=\mathrm{pr}_n\vert_D^{\sim}(y)$.
Then the trace function $I_\gamma$ associated to an element $\gamma$ of
 $\pi$ is holomorphic at $x$ 
if $\gamma$ fixes a certain vertex of the Bruhat--Tits building 
$\mathcal{B}_{n,\widetilde{D}, y}$ associated to $(\widetilde{D}, y)$.
\end{thm}

\begin{proof}
We first claim that $I_\gamma$ is holomorphic at $x$ if and only if 
$\tr \rho_{\widetilde{D}}(\gamma)$ is contained in the valuation ring
 $\mathcal{O}_y$ of $\C(D)$ with respect to the valuation $w_y=\mathrm{ord}_y$. 
Indeed we may easily check that $I_\gamma$ coincides with $\tr
 \rho_{\widetilde{D}}(\gamma)$ as an element of $\C(C) (\subset \C(D))$, 
 and the holomorphy of
 $I_\gamma$ at $x$ is equivalent to the non-negativity of the order of
 $I_\gamma$ at $x$. The claim easily follows from these observations 
combined with the elementary fact that, for each element $f$ of $\C(C)$,
the order $\mathrm{ord}_x(f)$ of $f$ at $x$
is non-negative if and
 only if $w_y(f)=\mathrm{ord}_y(f)$ is non-negative.

Let $v_0$ denote the vertex of $\mathcal{B}_{n,\widetilde{D}, y}$ 
represented by the standard lattice $\sum_{j=1}^n \mathcal{O}_y e_j$. 
The isotropy subgroup of $\mathrm{SL}_n(\C(D))$ at $v_0$ is then calculated as 
$Z(\mathrm{SL}_n(\C(D)))\mathrm{SL}_n(\mathcal{O}_y)$. Here $Z(\mathrm{SL}_n(\C(D)))$ denotes the
 centre of $\mathrm{SL}_n(\C(D))$ and consists of scalar matrices $aI_n$ 
where $a$ is an $n$\nobreakdash-th root of unity contained in $\C(D)$. 
But the group of $n$\nobreakdash-th roots of unity $\mu_n(\C(D))$
 contained in $\C(D)$ is indeed contained in $\mathcal{O}_y$ 
because $\mathcal{O}_y$ is integrally closed in $\C(D)$. Hence $Z(\mathrm{SL}_n(\C(D)))$
 is a subgroup of $\mathrm{SL}_n(\mathcal{O}_y)$ and the isotropic subgroup at
 $v_0$ exactly coincides with $\mathrm{SL}_n(\mathcal{O}_y)$. 

Now assume that $\gamma$ fixes a vertex $v$ of $\mathcal{B}_{n,
 \widetilde{D}, y}$. Then there exists an element $g$ of
 $\mathrm{Aut}_{\C(D)}(V_n)$ satisfying $gv_0=v$ (recall that $V_n$
 denotes the $n$-dimensional $\C(D)$-vector space $\sum_{j=1}^n
 \C(D)e_j$). The isotropic subgroup
 of $\mathrm{SL}_n(\C(D))$ at $v$ then coincides with
 $g\mathrm{SL}_n(\mathcal{O}_y)g^{-1}$, and hence $\rho_{\widetilde{D}}(\gamma)$
 is contained in the  conjugate $g\mathrm{SL}_n(\mathcal{O}_y)g^{-1}$ 
 of $\mathrm{SL}_n(\mathcal{O}_y)$. The trace function is invariant under
 conjugation, and we may thus conclude that $\tr
 \rho_{\widetilde{D}}(\gamma)$ is contained in $\mathcal{O}_y$ as desired.  
\end{proof}

As a direct consequence of Theorem~\ref{thm:FT}, we may verify that 
the action of $\pi$ associated to an {\em ideal point} of $X_n(\pi)$ is 
nontrivial.
Recall that an action of a group $G$ on a simplicial complex $\Delta$ is 
said to be {\em nontrivial} if, for every vertex $v$ of $\Delta$, 
the isotropic subgroup $G_v$ of $G$ at $v$ is a proper subgroup of $G$.

\begin{cor} \label{cor:nontrivial}
Let $\tilde{x}$ be an ideal point of an affine  curve $C$ contained in $X_n(\pi)$ and $\tilde{y}$ a lift
 of $\tilde{x}$ $($namely, an ideal point of a lift $D$ of $C$
 satisfying $\mathrm{pr}_n\vert_D^{\sim}(\tilde{y})=\tilde{x})$. 
Then the associated action of $\pi$ on
 $\mathcal{B}_{n,\widetilde{D}, \tilde{y}}$ is nontrivial.
\end{cor}

\begin{proof}
Let $D$ be a lift of $C$ in $R_n(\pi)$.
Striving for a contradiction, 
suppose that the action of $\pi$ induced on
 $\mathcal{B}_{n,\widetilde{D},y}$ is trivial, or in other words,
 suppose that  
there exists a vertex $v$ of $\mathcal{B}_{n,\widetilde{D},\tilde{y}}$ 
at which 
the isotropic subgroup of $\pi$ coincides 
with the whole group $\pi$. 
Theorem~\ref{thm:FT} then implies that 
the trace function $I_\gamma$ does not have a pole at
 $\tilde{x}$ for {\em every} element $\gamma$ of $\pi$.
In particular every affine coordinate function of $C$ is holomorphic at
 $\tilde{x}$ due to Theorem~\ref{thm:Procesi}. 
The last assertion contradicts the fact that  
 at least one coordinate function must have a pole at $\tilde{x}$ (recall
 that we have chosen $\tilde{x}$ from {\em ideal} points of $C$). 
\end{proof}

\begin{rem}
In the case where $n$ equals $2$, 
Culler and Shalen have also verified the converse 
of Theorem~\ref{thm:FT} in \cite[Theorem~2.2.1]{CS}; 
namely, they have proved that if $I_\gamma$ is 
holomorphic at $x$ (or equivalently, if $\tr
 \rho_{\widetilde{D}}(\gamma)$ is contained in $\mathcal{O}_y$), there
 exists a vertex of $\mathcal{B}_{n, \widetilde{D}, y}$ which is fixed
 by the action of $\gamma$. When $n$ is greater than or equal to $3$, 
 the converse of
 Theorem~\ref{thm:FT} does not hold at all in general 
(indeed one readily observes that the proof given in \cite[Theorem~2.2.1]{CS} 
clearly collapses for matrices of higher rank). 
Theorem~\ref{thm:FT} is, however, sufficient to construct 
tribranched surfaces, and the failure of the
 converse of Theorem~\ref{thm:FT} does not cause any harm 
 for the purpose of this article.
\end{rem}

\subsection{Ideal points of character varieties and tribranched surfaces}
\label{ssc:construction}

Now we show that an essential tribranched surface in a $3$-manifold
is constructed from a nontrivial type-preserving action
of its fundamental group on a Euclidean building.
Such an action is obtained from an ideal point of an affine curve in the character variety
as in Section \ref{ssc:NAB}, and consequently, an ideal point gives
an essential tribranched surface under certain conditions.

Let $M$ be a compact, connected, irreducible and orientable $3$-manifold.
In the following argument, a ``triangulation'' of $M$ should be understood to be {\em a piecewise-linear triangulation}, that is, the link of every $i$-simplex in the triangulation is piecewise-linearly homeomorphic to a $(2 - i)$-simplex or the boundary of a $(3- i)$-simplex for $i = 0, 1, 2$, according as the $i$-simplex lies in $\partial M$ or not  (see for instance \cite[Chapter 1]{Hempel}).
Now let $K$ be a (possibly locally infinite) $2$-dimensional combinatorial CW-complex each of whose closed cells is identified with a simplex as a CW-complex.
We call a map $f \colon M \to K$ \textit{piecewise-linear} if, for some triangulation of $M$, the images of the vertices of every simplex in $M$ span a simplex in $K$, and the restriction of $f$ to each simplex of $M$ is a linear map.
We define $Y(K)$ to be the $1$-dimensional subcomplex of the first barycentric subdivision of $K$ consisting of all the barycentres of $1$- and $2$-simplices and all the edges connecting them.

\begin{lem} \label{lem:Y}
Let $f \colon M \to K$ be a piecewise-linear map.
Then the inverse image of $Y(K)$ under $f$ is a tribranched surface in $M$.
\end{lem}

\begin{proof}
Consider a triangulation of $M$ with respect to which $f$ is a piecewise-linear map, and set $\Sigma$ to be the inverse image of $Y(K)$ under $f$.
Note that $\Sigma$ is a compact subset of $M$ since it is a closed subset of the compact manifold $M$. Let us use the notation introduced in Section~\ref{ssc:TBS}.

We first show that $(M, \Sigma)$ is locally homeomorphic to $(\overline{\mathbb{H}}, Y \times [0, \infty))$;
recall that we define the topological space $Y$ as 
\[ Y = \{ \, r e^{\sqrt{-1} \theta} \in \C \mid r\in \R_{\geq 0} \text{ and }\theta = 0, \pm 2 \pi/3 \, \}. \]
The piecewise-linear map $f$ maps each $3$-simplex $\tau$ in $M$ onto either a vertex, an edge or a $2$-simplex in $K$.
Corresponding to the image of $\tau$ in $K$, the restriction of $\Sigma$ to $\tau$ is either the empty set, a normal disk (more precisely a triangle or a quadrilateral), or a $2$-dimensional combinatorial CW-complex consisting of one triangle and two quadrilaterals sharing one common edge; see Figure~\ref{fig:Y}.
The inverse image $\Sigma$ is the union of these subspaces glued up along $2$-simplices in $M$.

\begin{figure}[ht]
\centering
\begin{tabular}{c@{\qquad}c@{\qquad}c}
  \begin{tikzpicture}
   \coordinate (O) at (0,0);
   \coordinate (A) at (-2.0,-2);
   \coordinate (B) at (0.4,-3.2);
   \coordinate (C) at (1.4,-1.2);
   \draw [darkgray, very thick, opacity=0.6] (O) -- (B) (O) -- (A) -- (B) -- (C) --cycle;
   \draw [darkgray, very thick, opacity=0.6, dashed] (A) -- (C);
   \coordinate (OA) at ($0.5*(A)$);
   \coordinate (OB) at ($0.5*(B)$);
   \coordinate (OC) at ($0.5*(C)$);
   \fill [pattern=north east lines] (OA) -- (OB) -- (OC) -- cycle;
   \draw (OA) -- (OB) -- (OC) -- cycle;
  \end{tikzpicture} &
  \begin{tikzpicture}
   \coordinate (O) at (0,0);
   \coordinate (A) at (-2.0,-2);
   \coordinate (B) at (0.4,-3.2);
   \coordinate (C) at (1.4,-1.2);
   \draw [darkgray, very thick, opacity=0.6] (O) -- (B) (O) -- (A) -- (B) -- (C) --cycle;
   \draw [darkgray, very thick, opacity=0.6, dashed] (A) -- (C);
   \coordinate (OA) at ($0.5*(A)$);
   \coordinate (OC) at ($0.5*(C)$);
   \coordinate (AB) at ($0.5*(A)+0.5*(B)$);
   \coordinate (BC) at ($0.5*(B)+0.5*(C)$);
   \fill [pattern=north east lines] (OA) -- (AB) -- (BC) -- (OC) -- cycle;
   \draw (OA) -- (AB) -- (BC) -- (OC) -- cycle;
  \end{tikzpicture} & 
 \begin{tikzpicture}
   \coordinate (O) at (0,0);
   \coordinate (A) at (-2.0,-2);
   \coordinate (B) at (0.4,-3.2);
   \coordinate (C) at (1.4,-1.2);
  \draw [darkgray, very thick, opacity=0.6] (O) -- (B) (O) -- (A) -- (B) -- (C) --cycle;
   \draw [darkgray, very thick, opacity=0.6, dashed] (A) -- (C);
  \coordinate (OA) at ($0.5*(A)$);
  \coordinate (OB) at ($0.5*(B)$);
  \coordinate (AB) at ($0.5*(A)+0.5*(B)$);
  \coordinate (BC) at ($0.5*(B)+0.5*(C)$);
  \coordinate (AC) at ($0.5*(A)+0.5*(C)$);
  \coordinate (OAB) at ($1/3*(A)+1/3*(B)$);
  \coordinate (ABC) at ($1/3*(A)+1/3*(B)+1/3*(C)$);
  \fill [pattern=north east lines] (OAB) -- (OB) -- (BC) -- (ABC) -- cycle;
  \draw (OAB) -- (OB) -- (BC) -- (ABC) -- cycle;
  \fill [pattern=north east lines] (OAB) -- (AB) -- (ABC) -- cycle;
  \draw (OAB) -- (AB) -- (ABC) -- cycle;
  \draw [name path=front] (OAB) -- (OB);
  \draw [opacity=0, name path=back] (AC) -- (ABC);
  \fill [pattern=north east lines, name intersections={of=front and back, by=P}] (OAB) -- (OA) -- (AC) -- (P) -- cycle;
  \draw (OAB) -- (OA) -- (AC) -- (P);
  \draw [dashed] (P) -- (ABC);
  \end{tikzpicture} \\
 \end{tabular}
\caption{The inverse image $\Sigma$ in a single $3$-simplex $\tau$ of $M$}
\label{fig:Y}
\end{figure}

Now take an arbitrary point $x$ of $\Sigma$, and let us study the topological structure around $x$. Firstly we know from the construction of $\Sigma$ that $x$ cannot be any vertex in $M$.
If $x$ is in the interior of a $3$-simplex $\tau$ in $M$, then the above classification of types of $\Sigma$ restricted to $\tau$ shows at once that $(M, \Sigma)$ is locally homeomorphic to $(\mathbb{H}, Y \times (0, \infty))$ around $x$.
Next suppose that $x$ is in the interior of a $2$-simplex in $M$.
Then the above classification again shows that, for each $3$-simplex $\tau$ containing the $2$-simplex under consideration, a sufficiently small open neighborhood of $x$ in $\Sigma \cap \tau$ is homeomorphic to $\R \times [0, \infty)$ or $Y \times [0, \infty)$.
Since every $2$-simplex is adjacent to at most two $3$-simplices in $M$, a sufficiently small open neighborhood of $x$ in $\Sigma$ is homeomorphic to $\R \times [0, \infty)$ or $Y \times [0, \infty)$ if $x \in \partial M$, and to $\R^2$ or $Y \times \R$ otherwise.
It thus follows that $(M, \Sigma)$ is locally homeomorphic to $(\overline{\mathbb{H}}, Y \times [0, \infty))$ around $x$ in this case.

As the final case suppose that $x$ is the midpoint of an edge in $M$.
Note that, for each $3$-simplex $\tau$ containing the edge under consideration, a sufficiently small open neighborhood of $x$ in $\Sigma \cap \tau$ is a sector in any cases of the above classification of $\Sigma \cap \tau$. Since we consider a piecewise-linear triangulation of $M$, a finite number of $3$-simplices are glued along $2$-simplices around every edge in $M$ so that its link in $M$ is homeomorphic to a closed interval or a circle according as the edge lies in $\partial M$ or not.
If we take a sufficiently small open neighborhood of $x$ in $\Sigma$, we now see that its boundary is homeomorphic to the link of the edge under consideration since the neighbourhood is just the union of sectors glued up along $2$-simplices; see Figure~\ref{fig:NBD}. This implies that, aroud $x$, $\Sigma$ is homeomorphic to $\R \times [0, \infty)$ or $\R^2$ according as $x$ is contained in $\partial M$ or not, and thus $(M, \Sigma)$ is locally homeomorphic to $(\overline{\mathbb{H}}, \R \times [0, \infty))$ in this case.
In summary, we see that $(M, \Sigma)$ is locally homeomorphic to $(\overline{\mathbb{H}}, Y \times [0, \infty))$ around each $x$ in $\Sigma$.

\begin{figure}[ht]
\centering
\begin{tikzpicture}
 \coordinate (O) at (0,0);
 \coordinate (OO) at (0,-3);
 \coordinate (X) at (0,-1.5);
 \coordinate (A) at (-1.5,-2);
 \coordinate (AA) at (1.5,-1.8);
 \coordinate (B) at (-1.4,-0.9);
 \coordinate (BB) at (1.6,-1);
 \coordinate (C) at (-0.7,-0.2);
 \coordinate (CC) at (0.9,-0.1);
 \draw [name path=BOO, opacity=0] (B) -- (OO);
 \draw [name path=COO, opacity=0] (C) -- (OO);
 \draw [name path=BBOO, opacity=0] (BB) -- (OO);
 \draw [name path=CCOO, opacity=0] (CC) -- (OO);
 \draw [very thick] (O) -- (OO);
 \draw [gray, thick, opacity=0.6, name path=OA] (O) -- (A) -- (OO);
 \draw [gray, thick, opacity=0.6, name path=OB,name intersections={of=OA and BOO, by=P}] (O) -- (B) -- (P);
 \draw [gray, thick, opacity=0.6, name path=OC,name intersections={of=OB and COO, by=Q}] (O) -- (C) -- (Q);
 \draw [gray, thick, opacity=0.6, name path=OAA] (O) -- (AA) -- (OO);
 \draw [gray, thick, opacity=0.6, name path=OBB,name intersections={of=OAA and BBOO, by=PP}] (O) -- (BB) -- (PP);
 \draw [gray, thick, opacity=0.6, name path=OCC,name intersections={of=OBB and CCOO, by=QQ}] (O) -- (CC) -- (QQ);
 \draw [gray, thick, opacity=0.6, dashed] (P) -- (OO) (Q) -- (OO) (PP) -- (OO) (QQ) -- (OO);
 \fill (X) circle [radius=2pt];
\coordinate (OOA) at ($(O)!.5!(A)$);
 \coordinate (R) at ($(X)!0.5!(OOA)$);
\coordinate (RR) at ($(X)!1.2!(R)$);
\coordinate (OB) at ($(O)!.5!(B)$);
 \coordinate (S) at ($(X)!0.5!(OB)$);
\coordinate (SS) at ($(X)!1.2!(S)$);
\coordinate (OC) at ($(O)!.5!(C)$);
\coordinate (T) at ($(X)!0.5!(OC)$);
\coordinate (TT) at ($(X)!1.2!(T)$);
\coordinate (OOAA) at ($(OO)!.5!(AA)$);
 \coordinate (U) at ($(X)!0.5!(OOAA)$);
\coordinate (UU) at ($(X)!1.2!(U)$);
\coordinate (OBB) at ($(O)!.5!(BB)$);
 \coordinate (V) at ($(X)!0.4!(BB)$);
\coordinate (VV) at ($(X)!1.2!(V)$);
\coordinate (OCC) at ($(O)!.4!(CC)$);
\coordinate (W) at ($(X)!0.5!(CC)$);
\coordinate (WW) at ($(X)!1.2!(W)$);
\draw [thin] (X) -- (RR) (X) -- (SS) (X) -- (TT) (X) -- (UU) (X) -- (VV) (X) -- (WW);
\fill [pattern=north east lines] (RR) -- (X) -- (SS) -- cycle (TT) -- (X) -- (WW) (VV) -- (X) -- (UU);
\fill [pattern=north west lines] (SS) -- (X) -- (TT) -- cycle (WW) -- (X) -- (VV);
\fill [white] ($(X)+(0.1,0.15)$) rectangle ($(X)+(0.4,-0.15)$);
\node at (X) [right] {$x$};
\draw [ultra thick] (A) -- (B) -- (C) (AA) -- (BB) -- (CC);
\draw [ultra thick, dashed] (C) -- (CC);
\draw [very thick] (R) -- (S) -- (T) -- (W) -- (V) -- (U);
\coordinate (M) at ($(CC)!.5!(BB)+(0.1,0)$);
\coordinate (MM) at (3,-0.8);
\draw [->] (MM) [bend right=20pt] to (M);
\node at (MM) [xshift=15pt,yshift=-8pt] {the link of the edge};
\coordinate (N) at ($(R)!.5!(S)+(-0.1,0)$);
\coordinate (NN) at (-3,-2);
\draw [->] (NN) [bend left=20pt] to (N);
\node at (NN) [yshift=-8pt] {the boundary of};
\node at (NN) [yshift=-23pt] {a neighbourhood of $x$ in $\Sigma$};
\coordinate (D) at (0.1,-2.25);
\coordinate (DD) at (2,-2.5);
\draw [->] (DD) [bend right=20pt] to (D);
\node at (DD) [below,xshift=18pt] {the edge containing $x$};
\end{tikzpicture} 
\caption{Around the midpoint $x$ of an edge}
\label{fig:NBD}
\end{figure}

Next we show that $\Sigma$ satisfies (TBS1).
Let $C$ be an arbitrary component of the set $C(\Sigma)$ of branched points, and consider a sufficiently small tubular neighbourhood $\nu(C)$ of $C$ in $M$.
The intersection $\nu(C) \cap \Sigma$ naturally admits the structure of a fibre bundle over $C$ whose fibre is homeomorphic to $Y$.
We may identify $f(\nu(C) \cap \Sigma)$ with $Y$ so that $f(C)$ corresponds to $\{0\}$.
Then since the inverse image of $\{0\}$ under $f$ is $C$, the topological space $f((\nu(C) \cap \Sigma) \setminus C)$ has 3~components, and so does $(\nu(C) \cap \Sigma) \setminus C$ by continuity of $f$. Therefore the fibre bundle $\nu(C) \cap \Sigma \to C$ above must be trivial, which implies that $\Sigma$ satisfies (TBS1).

Finally  we show that $\Sigma$ satisfies (TBS2).
We denote by $M_0$ the complement of a small open tubular neighbourhood of $C(\Sigma)$ in $M$.
Let $S$ be an arbitrary component of the subsurface $S(\Sigma)$, which can be regarded as a properly embedded subsurface in the orientable $3$-manifold $M_0$.
The image $f(S)$ is contained in a component $\Gamma$ of the complement in $Y(K)$ of the subset consisting of all the barycentres of $2$-simplices.
Since $\Gamma$ is bicollared in $K$, $S$ is two-sided, and so orientable.
Hence $\Sigma$ satisfies (TBS2), and the lemma follows.
\end{proof}

We now consider a type-preserving action of $\pi_1(M)$ on a Euclidean building $\mathcal{B}$.
The simplicial complex structure of $\mathcal{B}^{(2)}$ naturally induces the combinatorial CW-complex structure of $\mathcal{B}^{(2)}/\pi_1(M)$, where, for each non-negative integer $i$, we denote by $\mathcal{B}^{(i)}$ the $i$-skeleton of $\mathcal{B}$.
In particular, each closed cell of the combinatorial CW-complex $\mathcal{B}^{(2)}/\pi_1(M)$ is identified with a simplex as a CW-complex.
We say that a type-preserving action of $\pi_1(M)$ on a Euclidean  building $\mathcal{B}$ {\em gives a tribranched surface $\Sigma$} if there exists a map $f \colon M \to \mathcal{B} / \pi_1(M)$ such that the tribranched surface $\Sigma$ coincides with the inverse image of $Y(\mathcal{B}^{(2)} / \pi_1(M))$ under $f$.

\begin{thm} \label{thm:BT}
Let $n$ be a natural number greater than or equal to $3$, 
and assume that the boundary $\partial M$ of $M$ is non-empty when 
$n$ is strictly greater than $3$.
Then a nontrivial type-preserving action of $\pi_1(M)$
on a Euclidean building $\mathcal{B}$ of dimension~$n-1$ gives
an essential tribranched surface in $M$.
\end{thm}

\begin{proof}
The proof is divided into two parts.
In the first part we show that the action of $\pi_1(M)$ on $\mathcal{B}$ gives
a non-empty tribranched surface which is not necessarily essential,
and in the second part we modify such a tribranched surface given by the action
to be essential by local surgery.

Let us take a triangulation of $M$ and consider the triangulation on
 $\widetilde{M}$ induced from it.
We construct a $\pi_1(M)$\nobreakdash-equivariant simplicial map
$\tilde{f} \colon \widetilde{M} \to \mathcal{B}^{(2)}$ as follows.
First consider the case of $n=3$ (in the case the
 $2$-skeleton of $\mathcal{B}$ coincides with $\mathcal{B}$ itself since 
it is of dimension $2$).
For each vertex $v$ of $M$, we choose a lift $\tilde{v}$ of $v$ in $\widetilde{M}$
and a vertex $\tilde{f}(\tilde{v})$ of $\mathcal{B}$.
Then we define $\tilde{f}|_{\widetilde{M}^{(0)}}$ as
\[ \tilde{f}\vert_{\widetilde{M}^{(0)}}(\gamma \cdot \tilde{v}) = \gamma \tilde{f}(\tilde{v}) \]
for arbitrary $\gamma \in \pi_1(M)$ so that
$\tilde{f}|_{\widetilde{M}^{(0)}}$ is $\pi_1(M)$-equivariant.
Now assume that we have already constructed a $\pi_1(M)$-equivariant
 simplicial map $\tilde{f}\vert_{\widetilde{M}^{(i-1)}} \colon
 \widetilde{M}^{(i-1)}\rightarrow \mathcal{B}$ on the
 $(i-1)$\nobreakdash-skeleton of $\widetilde{M}$, and let us take an
 arbitrary $i$-simplex $\sigma$ of $\widetilde{M}$. 
We may extend the restriction
 $\tilde{f}\vert_{\partial \sigma}$ of
 $\tilde{f}\vert_{\widetilde{M}^{(i-1)}}$ onto $\partial \sigma$ to a
 map $\tilde{f}\vert_\sigma$ on $\sigma$ due to the contractibility of
 the Euclidean building $\mathcal{B}$. Moreover 
we can take $\tilde{f}|_{\sigma}$ to be a simplicial map by subdividing
 $M$ (and $\widetilde{M}$) if necessary.
By continuing this procedure, we can extend
 $\tilde{f}|_{\widetilde{M}^{(0)}}$ to simplicial maps 
on $\widetilde{M}^{(1)}$, $\widetilde{M}^{(2)}$,
and $\widetilde{M}$ inductively, and obtain a desired simplicial map 
$\tilde{f}\colon \widetilde{M}\to \mathcal{B}$.
Next consider the case of $n \geq 4$.
Since $\partial M$ is non-empty by the assumption, we can take a $2$-dimensional subcomplex $V$ which is a deformation retract of $M$.
Denote by $\widetilde{V}$ the preimage of $V$ under the universal covering map $\widetilde{M} \to M$.
We define $\tilde{f}\vert_{\widetilde{V}^{(0)}}$ on
 $\widetilde{V}^{(0)}$ and, by subdividing $M$ if necessary, 
we extend it to a $\pi_1(M)$-equivariant simplicial map 
$\tilde{f}\vert_{\widetilde{V}}\colon \widetilde{V}\rightarrow \mathcal{B}^{(2)}$ similarly to the
 case of $n=3$. Note that the image of the extended map
 $\tilde{f}\vert_{\widetilde{V}}$ is contained in the
 $2$\nobreakdash-skeleton $\mathcal{B}^{(2)}$ of $\mathcal{B}$ since
 $\widetilde{V}$ is of dimension~$2$.
By composing $\tilde{f}\vert_{\widetilde{V}}$  
with a deformation retraction $\widetilde{M} \to \widetilde{V}$, 
we obtain a desired map $\tilde{f} \colon \widetilde{M} \to \mathcal{B}^{(2)}$.

We can slightly modify the above construction so that the restriction
of $\tilde{f} \colon \widetilde{M} \to \mathcal{B}^{(2)}$ to each simplex
is a linear map.
Denote by $f \colon M \to \mathcal{B}^{(2)} / \pi_1(M)$ the quotient of the simplicial map
$\tilde{f} \colon \widetilde{M} \to \mathcal{B}^{(2)}$ by $\pi_1(M)$,
which is a piecewise-linear map,
 and set $\Sigma$ to be the inverse image of $Y(\mathcal{B}^{(2)} / \pi_1(M))$ under $f$.
By Lemma~\ref{lem:Y} we see that $\Sigma$ is a tribranched surface in $M$.
Note that the above construction of $\Sigma$ is far from being canonical
since it depends on many choices, for instance, of a triangulation of $M$
and a $\pi_1(M)$-equivariant simplicial map $\tilde{f}$.

Next we show that $\Sigma$ satisfies (ETBS1),
which, in particular, implies that $\Sigma$ is non-empty.
Striving for a contradiction, suppose that 
there exists a component $N$ of $M(\Sigma)$
such that the homomorphism $\pi_1(N) \to \pi_1(M)$ induced
by the natural inclusion $N\hookrightarrow M$ is surjective.
Let $N_0$ be a component of the preimage of $N$
under the universal covering map $\widetilde{M} \to M$.
Since $\tilde{f}(N_0)$ does not intersect $Y(\mathcal{B}^{(2)})$ 
by construction,
it is contained in the open star of a certain vertex $v$ of $\mathcal{B}^{(2)}$
in its barycentric subdivision. 
Obviously $N_0$ is a covering space over $N$, and thus the fundamental
 group $\pi_1(N)$ stabilises $N_0$. 
The image of the homomorphism $\pi_1(N) \to \pi_1(M)$ then also
stabilises the open star of $v$ containing $\tilde{f}(N_0)$ due to 
the $\pi_1(M)$-equivariance of $\tilde{f}$, and it is, in particular, 
contained in the isotropic subgroup $\pi_1(M)_v$ of $\pi_1(M)$ at $v$. 
Hence we conclude that $\pi_1(M)_v$ coincides with the whole group
 $\pi_1(M)$, combining the arguments above with 
the assumption on the surjectivity of the homomorphism $\pi_1(N)\to
 \pi_1(M)$, which contradicts nontriviality of the action of $\pi_1(M)$
 on $\mathcal{B}$.

As we have already mentioned at the beginning of the proof,
the tribranched surface $\Sigma$ itself might not be essential.
From now on we modify $\Sigma$ to be essential
as the second part of the proof.
For a tribranched surface $\Sigma$ given by the action of $\pi_1(M)$
on $\mathcal{B}$, we set
\begin{align*}
l(\Sigma) &= \text{the number of components of $C(\Sigma)$}, \\
m(\Sigma) &= \sum_S (2-\chi(S))^2 \quad (\text{where $\chi(S)$ is the Euler characteristic of the surface $S$}), \\
n(\Sigma) &= \text{the number of components of $\Sigma$},
\end{align*}
where the sum in the second equation runs over all components $S$ of $S(\Sigma)$.
We see at once that these integers are all non-negative.
We consider the triple $(l(\Sigma), m(\Sigma), n(\Sigma)) \in \Z^3$
with respect to the lexicographical order of $\Z^3$
as a {\em complexity} of a non-empty tribranched surface $\Sigma$.
In the following we show that if $\Sigma$ is not essential,
there are operations of replacing $\Sigma$ by another tribranched surface
with lower complexity, which is also given by the actions of $\pi_1(M)$ on $\mathcal{B}$. Consequently a tribranched surface of minimal complexity given by the
 action of $\pi_1(M)$ on $\mathcal{B}$
must be essential.

Let us consider the case where $\Sigma$ does not satisfy (ETBS2).
First assume that there exists a pair of components $C$ and $S$
of $C(\Sigma)$ and $S(\Sigma)$ respectively such that the natural inclusion map
between them induces a homomorphism $\pi_1(C) \to \pi_1(S)$
which is not injective.
This implies that $S$ is a disk.
Let $S_1$ and $S_2$ be the other components of $S(\Sigma)$ whose boundary
contain parallel copies of $C$ as components (the surfaces $S_1$ and
 $S_2$ might coincide). 
Take a small neighbourhood $B$ of $S$ which is homeomorphic to a ball
and intersects $S_1$ and $S_2$ in the collars of $C$.
Figure~\ref{fig:B1} illustrates a local picture of the neighbourhood $B$.
\begin{figure}[ht]
\centering
\begin{tikzpicture}
\draw (-4,3) .. controls (0,1.25) and (0,1.25) .. (4,3);
\draw (-4,-3) .. controls (0,-1.25) and (0,-1.25) .. (4,-3);

\fill [pattern=north west lines] (0,-1.68) arc (-90:90:0.4 and 1.68) arc (90:270:0.4 and 1.68);
\draw [very thick](0,-1.68) arc (-90:90:0.4 and 1.68);
\draw [very thick,dashed] (0,1.68) arc (90:270:0.4 and 1.68);
\draw (-1.5,-1.97) arc (-90:90:0.4 and 1.97);
\draw[dashed] (-1.5,1.97) arc (90:270:0.4 and 1.97);
\draw (1.5,-1.97) arc (-90:90:0.4 and 1.97);
\draw[dashed] (1.5,1.97) arc (90:270:0.4 and 1.97);

\draw (0,0) circle (2.5);

\draw (-1.5,-2.5) -- (-1.5,-3) -- (-0.5,-3);
\draw (0.5,-3) -- (1.5,-3) -- (1.5,-2.5);

\draw (-3,3) node {$\Sigma$};
\draw (0,2) node {$C$};
\draw (0,0) node [rectangle, fill=white] {$S$};
\draw (0,3) node {$B$};
\draw (-1.5,0) node {$D_1$};
\draw (1.5,0) node {$D_2$};
\draw (-2.2,0) node {$B_1$};
\draw (2.2,0) node {$B_2$};
\draw (-4,0) node {$S_1$};
\draw (4,0) node {$S_2$};
\draw (0,-3) node {$B_3$};
\end{tikzpicture}
\caption{A neighbourhood $B$ of the surface component $S$}
\label{fig:B1}
\end{figure}
Choose properly embedded disks $D_1$ and $D_2$ in $B$
bounding $S_1 \cap \partial B$ and $S_2 \cap \partial B$ respectively and not intersecting $S$.
We construct a map $g \colon B \to \mathcal{B}^{(2)} / \pi_1(M)$
such that $g|_{\partial B} = f|_{\partial B}$
and that $g^{-1}(Y(\mathcal{B}^{(2)} / \pi_1(M))) = D_1 \cup D_2$ as follows.
Since $f(\partial D_1)$ and $f(\partial D_2)$ are contained in open edges
of $Y(\mathcal{B}^{(2)} / \pi_1(M))$ near the vertex $f(C)$, the maps 
$g|_{\partial D_1}=f|_{\partial D_1}$ and $g|_{\partial D_2}=f|_{\partial D_2}$ extend to $D_1$ and $D_2$ respectively
so that $g(D_1)$ and $g(D_2)$ are contained in the same open edges.
The ball $B$ is divided into $3$ balls $B_1$, $B_2$ and $B_3$ by $D_1$ and $D_2$,
where $\partial B_1$ does not contain $D_2$ but contains $D_1$,
$\partial B_2$ does not contain $D_1$ but contains $D_2$,
and $\partial B_3$ contains both disks.
There exists a unique $2$-simplex of $\mathcal{B}^{(2)}$ which contains
 $f(C)$ as its barycentre, and the open star of each of its 3~vertices 
contains one of $g(\partial B_1 \setminus D_1)$, $g(\partial B_2
 \setminus D_2)$ and $g(\partial B_3 \setminus (D_1 \cup D_2))$. 
We can thus extend $g|_{\partial B_1}$, $g|_{\partial B_2}$ and $g|_{\partial B_3}$
to $B_1$, $B_2$ and $B_3$ respectively
so that all of $g(B_1 \setminus D_1)$, $g(B_2 \setminus D_2)$ and $g(B_3 \setminus (D_1 \cup D_2))$ do not intersect $Y(\mathcal{B}^{(2)} / \pi_1(M))$.
Then we see at once that the inverse images of $Y(\mathcal{B}^{(2)} /
 \pi_1(M))$ under the maps $g\vert_{B_1}$, $g\vert_{B_2}$ and
 $g\vert_{B_3}$ are $D_1$, $D_2$
and $D_1 \cup D_2$ respectively.
Figure \ref{fig:G1} illustrates the image $g(B)$ of the neighbourhood
 $B$ of $S$. 
\begin{figure}[ht]
\centering
\begin{tikzpicture}
\draw (0,4) -- (-30:4) -- (-150:4) -- (0,4);
\draw [dashed] (0,0) -- (30:2);
\draw [dashed] (0,0) -- (150:2);
\draw [dashed] (0,0) -- (0,-2);

\fill (0,0) circle [radius=2pt];
\node at (0,0) [below] {$f(C)$};

\draw[line width=1.5pt] (30:0.5) -- (30:1.5);
\draw[line width=1.5pt] (150:0.5) -- (150:1.5);
\draw[line width=1.5pt] (30:0.5) arc (0:180:0.43);
\draw[line width=1.5pt] (30:1.5) .. controls (0,3) and (0,3) .. (150:1.5);
\draw[line width=1.5pt] (30:0.5) .. controls (0.43,-2) and (2.88,-2) .. (30:1.5);
\draw[line width=1.5pt] (150:0.5) .. controls (-0.43,-2) and (-2.88,-2) .. (150:1.5);

\fill[gray, opacity=0.5] (30:0.5) -- (30:1.5) .. controls (0,3) and (0,3) .. (150:1.5) -- (150:0.5) arc (180:0:0.43);
\fill[gray, opacity=0.5] (150:1.5) -- (150:0.5) .. controls (-0.43,-2) and (-2.88,-2) .. (150:1.5);
\fill[gray, opacity=0.5] (30:1.5) -- (30:0.5) .. controls (0.43,-2) and (2.88,-2) .. (30:1.5);

\draw[->,>=latex] (3.5,0) -- (30:1);
\draw[->,>=latex] (-3.5,0) -- (150:1);

\draw (-1.2,-0.5) node {$g(B_1)$};
\draw (1.2,-0.5) node {$g(B_2)$};
\draw (0,1.3) node {$g(B_3)$};
\draw (-4.1,0) node {$g(D_1)$};
\draw (4.1,0) node {$g(D_2)$};
\end{tikzpicture}
\caption{The image $g(B)$ for the pair $(C, S)$}
\label{fig:G1}
\end{figure}
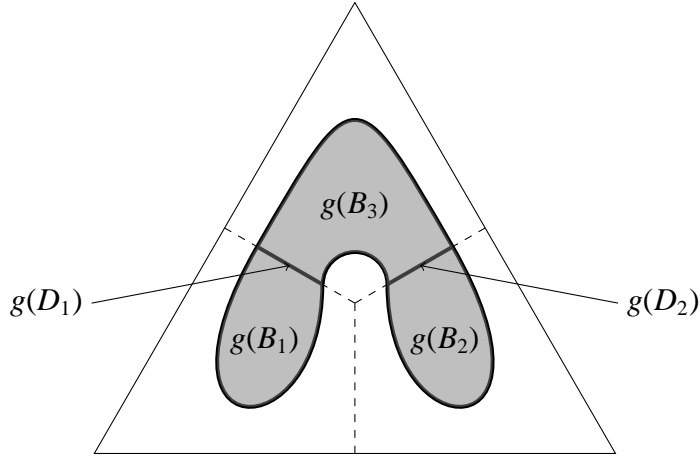
We now define $f' \colon M \to \mathcal{B}^{(2)} / \pi_1(M)$
so that $f'|_{M \setminus B} = f|_{M \setminus B}$ and $f'|_B = g$.
Then $f'^{-1}(Y(\mathcal{B}^{(2)} / \pi_1(M)))$ is another tribranched surface
and has a lower complexity since $l(\Sigma)$ decreases.

Next assume that there exists a pair of components $S$ and $N$
of $S(\Sigma)$ and $M(\Sigma)$ respectively such that the natural inclusion map
between them induces a homomorphism $\pi_1 (S) \to \pi_1 (N)$
which is not injective.
By Dehn's lemma, there exits a compressing disk $D$ of $S$ in $N$.
Take a small neighbourhood $B$ of $D$ which is homeomorphic to a ball
and intersects an annulus in $S$.
Figure~\ref{fig:B2} illustrates a local picture of the neighbourhood $B$.
\begin{figure}[ht]
\centering
\begin{tikzpicture}
\fill [pattern=crosshatch,opacity=.3] (-4,3) .. controls (0,1.25) and (0,1.25) .. (4,3) -- (4,-3) .. controls (0,-1.25) and (0,-1.25) .. (-4,-3) -- (-4,3); 
\draw (-4,3) .. controls (0,1.25) and (0,1.25) .. (4,3);
\draw (-4,-3) .. controls (0,-1.25) and (0,-1.25) .. (4,-3);

\draw [very thick](0,-1.68) arc (-90:90:0.4 and 1.68);
\draw[very thick,dashed] (0,1.68) arc (90:270:0.4 and 1.68);
\draw (-1.5,-1.97) arc (-90:90:0.4 and 1.97);
\draw[dashed] (-1.5,1.97) arc (90:270:0.4 and 1.97);
\draw (1.5,-1.97) arc (-90:90:0.4 and 1.97);
\draw[dashed] (1.5,1.97) arc (90:270:0.4 and 1.97);

\draw (0,0) circle (2.5);

\draw (-1.5,-2.5) -- (-1.5,-3) -- (-0.5,-3);
\draw (0.5,-3) -- (1.5,-3) -- (1.5,-2.5);

\draw (-3,3) node {$S$};
\draw (0,0) node {$D$};
\draw (0,3) node {$B$};
\draw (-1.5,0) node {$D_1$};
\draw (1.5,0) node {$D_2$};
\draw (-2.2,0) node {$B_1$};
\draw (2.2,0) node {$B_2$};
\draw (0,-3) node {$B_3$};
\node [rectangle, fill=white] at (3.5,-2) {$N$};
\end{tikzpicture}
\caption{A neighbourhood $B$ of the compression disk $D$}
\label{fig:B2}
\end{figure}
Choose properly embedded disks $D_1$ and $D_2$ in $B$
bounding the components of the boundary of the annulus.
We construct a map $g \colon B \to \mathcal{B}^{(2)} / \pi_1(M)$
such that $g|_{\partial B} = f|_{\partial B}$
and that $g^{-1}(Y(\mathcal{B}^{(2)} / \pi_1(M))) = D_1 \cup D_2$ as follows.
Since $f(\partial D_1)$ and $f(\partial D_2)$ are contained in the open star
of a vertex in $Y(\mathcal{B}^{(2)} / \pi_1(M))$ which is a barycentre of an edge of $\mathcal{B}$,
the maps $g|_{\partial D_1}=f|_{\partial D_1}$ and $g|_{\partial D_2}=f|_{\partial D_2}$ extend to $D_1$ and $D_2$
respectively so that $g(D_1)$ and $g(D_2)$ are contained in the same star.
The ball $B$ is divided into $3$ balls $B_1$, $B_2$ and $B_3$ by $D_1$ and $D_2$,
where $\partial B_1$ does not contain $D_2$ but contains $D_1$,
$\partial B_2$ does not contain $D_1$ but contains $D_2$,
and $\partial B_3$ contains both disks.
Since $g(\partial B_1 \setminus D_1)$ and $g(\partial B_2 \setminus D_2)$ are contained in the open star of a vertex of $\mathcal{B}^{(2)}$ in its barycentric subdivision (corresponding to $N$), and since $g(\partial B_3 \setminus (D_1 \cup D_2))$ is contained in that of another vertex of $\mathcal{B}^{(2)}$,
we can extend $g|_{\partial B_1}$, $g|_{\partial B_2}$ and $g|_{\partial B_3}$
to $B_1$, $B_2$ and $B_3$ respectively
so that $g(B_1 \setminus D_1)$, $g(B_2 \setminus D_2)$ and $g(B_3 \setminus (D_1 \cup D_2))$ do not intersect $Y(\mathcal{B}^{(2)} / \pi_1(M))$.
Then we see at once that the inverse images of $Y(\mathcal{B}^{(2)} /
 \pi_1(M))$ under the maps $g\vert_{B_1}$, $g|_{B_2}$
and $g|_{B_3}$ are $D_1$, $D_2$ and $D_1\cup D_2$ respectively.
Figure \ref{fig:G2} illustrates the image $g(B)$ of the neighbourhood
 $B$ of the compression disk $D$.
\begin{figure}[ht]
\centering
\begin{tikzpicture}
\draw (0,3) -- (0,1.7);
\draw (0,-3) -- (0,0.3);
\draw (0,3) -- (5.2,0) -- (0,-3) -- (-5.2,0) -- (0,3);

\fill [pattern=crosshatch,opacity=.3]  (1.73,0) -- (-1.73,0) -- (-2.59,-1.5) -- (0,-3) -- (2.59,-1.5) -- cycle;

 \draw [->,>=latex] (2,-2.4) .. controls (1.5,-2.6) and (1.2,-2.6) .. (0.5,-2.2);
 \node at (2,-2.4) [right] {\parbox{8em}{a star \\ corresponding to $N$}};
 
\draw [dashed] (1.73,0) -- (0,0);
\draw [dashed] (1.73,0) -- (2.59,1.5);
\draw [dashed] (1.73,0) -- (2.59,-1.5);
\draw [dashed] (-1.73,0) -- (0,0);
\draw [dashed] (-1.73,0) -- (-2.59,1.5);
\draw [dashed] (-1.73,0) -- (-2.59,-1.5);

\draw[line width=1.5] (0.3,0) -- (1.3,0);
\draw[line width=1.5] (-0.3,0) -- (-1.3,0);
\draw[line width=1.5] (0.3,0) arc (0:180:0.3);
\draw[line width=1.5] (1.3,0) .. controls (0,2.25) and (0,2.25) .. (-1.3,0);
\draw[line width=1.5] (0.3,0) .. controls (0.43,-2.83) and (2.17,-1.5) .. (1.3,0);
\draw[line width=1.5] (-0.3,0) .. controls (-0.43,-2.83) and (-2.17,-1.5) .. (-1.3,0);

\fill[gray, opacity=0.5] (0.3,0) -- (1.3,0) .. controls (0,2.25) and (0,2.25) .. (-1.3,0) -- (-0.3,0) arc (180:0:0.3);
\fill[gray, opacity=0.5] (-1.3,0) -- (-0.3,0) .. controls (-0.43,-2.83) and (-2.17,-1.5) .. (-1.3,0);
\fill[gray, opacity=0.5] (1.3,0) -- (0.3,0) .. controls (0.43,-2.83) and (2.17,-1.5) .. (1.3,0);

\fill (0,-3) circle [radius=2pt];

\draw[->,>=latex] (-2.3,2.7) -- (-0.8,0.05);
\draw[->,>=latex] (2.3,2.7) -- (0.8,0.05);

\draw (-0.98,-0.7) node {$g(B_1)$};
\draw (0.98,-0.7) node {$g(B_2)$};
\draw (0,0.8) node {$g(B_3)$};
\draw (-2.3,3) node {$g(D_1)$};
\draw (2.3,3) node {$g(D_2)$};
\end{tikzpicture}
\caption{The image $g(B)$ for the pair $(S, N)$}
\label{fig:G2}
\end{figure}
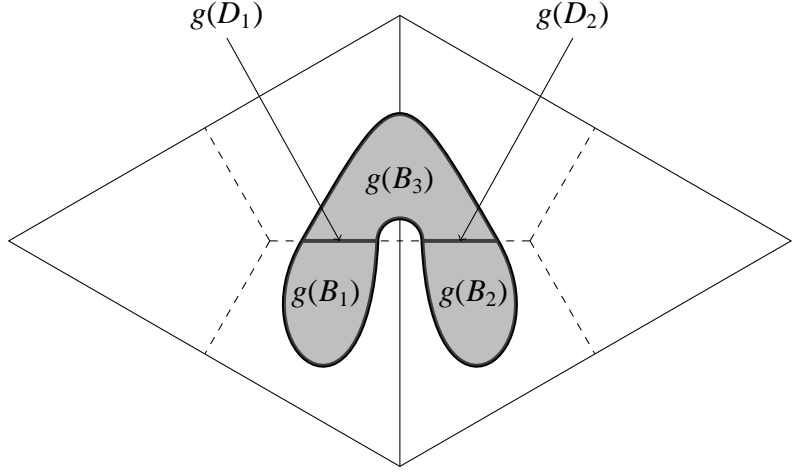
Now we define $f' \colon M \to \mathcal{B}^{(2)} / \pi_1(M)$
so that $f'|_{M \setminus B} = f|_{M \setminus B}$ and $f'|_B = g$.
Set $\Sigma' = f'^{-1}(Y(\mathcal{B}^{(2)} / \pi_1(M)))$,
which is another tribranched surface with the same $l(\Sigma')$ as $l(\Sigma)$.
We show in the followings that $m(\Sigma') $ is strictly less than $m(\Sigma)$,
which implies that $\Sigma'$ has a lower complexity than $\Sigma$.
Set $S' = (S \setminus B) \cup D_1 \cup D_2$.
First suppose that $S'$ is connected.
Then we can calculate as 
\begin{align*}
m(\Sigma) - m(\Sigma') &= (2-\chi(S))^2-(2-\chi(S'))^2 \\
&= 4+4(2-\chi(S')) > 0
\end{align*}
by using $\chi(S') = \chi(S) + 2 \leq 2$.
Next suppose that $S'$ has two components $S_1'$ and $S_2'$.
Note that neither $S_1'$ nor $S_2'$ is a sphere. 
Then we can calculate as
\begin{align*}
m(\Sigma) - m(\Sigma') &= (2-\chi(S))^2-(2-\chi(S_1'))^2-(2-\chi(S_2'))^2 \\
&= 2(2-\chi(S_1'))(2-\chi(S_2')) > 0
\end{align*}
by using $\chi(S_1') + \chi(S_2') = \chi(S) + 2$, $\chi(S_1') < 2$ and
$\chi(S_2') < 2$. In both the cases $m(\Sigma')$ decreases from
$m(\Sigma)$, as desired.

Finally, we consider the case where $\Sigma$ does not satisfy (ETBS3).
Then we see as follows that, after eliminating a component of $\Sigma$
contained in a ball in $M$ or a collar of $\partial M$, the resultant tribranched surface is
also given by the action of $\pi_1(M)$ on $\mathcal{B}$.
If there is a component of $\Sigma$ contained in a ball $B$,
we can construct a map $f' \colon M \to \mathcal{B}^{(2)} / \pi_1(M)$
such that $f'|_{M \setminus B} = f|_{M \setminus B}$
and that $f'(B)$ does not intersect $Y(\mathcal{B}^{(2)} / \pi_1(M))$, because  $f(\partial B)$ is contained in a contractible component
of the complement of $Y(\mathcal{B}^{(2)} / \pi_1(M))$ in $\mathcal{B}^{(2)} / \pi_1(M)$.
If there is one contained in a collar of $\partial M$,
we set $f' \colon M \to \mathcal{B}^{(2)} / \pi_1(M)$ to be 
the composition of a deformation retraction
from $M$ to the complement of the collar with the restriction of $f$ to it.
In both the cases the complexity of a new tribranched surface
 defined as
$f'^{-1}(Y(\mathcal{B}^{(2)} / \pi_1(M)))$
is lower than the original one's,
since $l(\Sigma)$ and $m(\Sigma)$ do not increase and $n(\Sigma)$ decreases.
The proof is now completed.
\end{proof}

\begin{rem}
Since a tribranched surface of minimal complexity given
by the action of $\pi_1(M)$ on a Euclidean building is not necessarily unique,
the construction of an essential tribranched surface in the proof
is far from being canonical.
\end{rem}

Now let us return to the settings in Section~\ref{ssc:NAB}. 
Let $\tilde{x}$ be an ideal point of a curve $C$ in $X_n(M)$
and let $\tilde{y}$ be a lift of $\tilde{x}$, which is an ideal point of
a lift $D$ of $C$.
We say that {\em $\tilde{x}$ gives an tribranched surface $\Sigma$}
if the associated action of $\pi_1(M)$ on $\mathcal{B}_{n, \widetilde{D}, \tilde{y}}$
gives $\Sigma$.
The following is the main theorem of this article,
which is now a direct consequence
of Corollary~\ref{cor:nontrivial} and Theorem~\ref{thm:BT}.

\begin{thm} \label{thm:main}
Let $n$ be a natural number greater than or equal to $3$, 
and assume that the boundary $\partial M$ of $M$ is non-empty when 
$n$ is strictly greater than $3$.
Then an ideal point of an affine algebraic curve in $X_n(M)$ 
gives an essential tribranched surface in $M$. 
\end{thm}

\section{An application to small Seifert manifolds} \label{sc:A}

One of great advantages of extending Culler--Shalen theory to higher dimensional
representations is that we may apply our extended theory 
also to a {\em non-Haken}
$3$-manifold, that is, a $3$-manifold which does not contain any essential surfaces.
Here we describe an application of Theorem \ref{thm:main}
to a class of $3$-manifolds called {\em small Seifert manifolds}, which
contain non-Haken $3$-manifolds. We remark that all the homology
groups appearing in this section are singular homology groups.

A {\em Seifert manifold} is a compact, orientable $3$-manifold
admitting the structure of a Seifert fibred space
whose base orbifold is a compact surface with cone points.
A \textit{small Seifert manifold} is a Seifert manifold with at most $3$ singular fibres.
We refer the reader to \cite[Chapter IV]{Jaco}
for details on Seifert manifolds.

Let $p$, $q$ and $r$ be natural numbers greater than or equal to $3$. 
We denote by $S^2(p,q,r)$ the $2$-sphere with three cone points
whose cone angles are $2 \pi / p$, $2 \pi / q$ and $2 \pi / r$ respectively, and 
consider a small Seifert manifold $M$ with the base orbifold  $S^2(p, q, r)$.
Such a $3$-manifold is known to be irreducible, and it is Haken
if and only if its first homology group $H_1(M,\Z)$ is infinite.
The fundamental group $\pi_1(M)$ has a presentation of the form
\[ \langle x, y, h \mid h \text{: central}, x^p = h^a, y^q = h^b, (xy)^r = h^c \rangle \]
for certain integers $a, b, c$ satisfying $(a, p) = (b, q) = (c, r) = 1.$
The orbifold fundamental group $\pi_1^\mathrm{orb} (S^2(p, q, r))$ of
$S^2(p,q,r)$ is isomorphic
to the {\em ordinary triangle group} (or the {\em von Dyck group}) $\Delta(p, q, r)$ defined as
\[ \langle x, y \mid x^p = y^q = (xy)^r = 1 \rangle, \]
and by identifying $\pi_1^\mathrm{orb}(S^2(p,q,r))$ with $\Delta(p,q,r)$, 
we may regard the natural homomorphism $\pi_1(M) \to \pi_1^\mathrm{orb} (S^2(p, q, r))$
induced by the projection $M\rightarrow S^2(p,q,r)$ as the group homomorphism which 
maps $x$ and $y$ identically and sends $h$ to the unit (in particular it
is a surjection). 
It is easy to see that the first homology group $H_1(M,\Z)$ is infinite
if and only if the equality
\[ \frac{a}{p} + \frac{b}{q} = \frac{c}{r} \]
holds. From this observation we thus find that $M$ tends to be non-Haken in
most cases.

In the case where $M$ is Haken, 
we may readily construct an affine curve in $X_3(M)$ consisting 
of abelian characters because the first homology group 
$H_1(M,\Z)$ is infinite. In the following we verify that 
$X_3(M)$ contains an affine curve 
also in the case where $M$ is non-Haken.
It thus follows from Theorem~\ref{thm:main} that an ideal point
of the curve gives an essential tribranched surface $\Sigma$ contained
in $M$, which one can never obtain by utilising classical Culler--Shalen theory
(since the $\mathrm{SL}_2(\C)$-character variety $X_2(M)$ is of dimension $0$ in
the case).

We may regard the group $\Delta(p, q, r)$ as a subgroup of index $2$
of the {\em Schwartzian triangle group} $\Gamma(p, q, r)$ defined as
\[  \langle a, b, c  \mid a^2 = b^2 = c^2 = (ab)^p = (bc)^q = (ca)^r = 1 \rangle, \]
identifying $x$ with $ab$ and $y$ with $bc$
respectively.
It follows from the argument in \cite[Section 6]{Goldman}
that there exists a family of $\mathrm{SL}_3(\C)$-representations 
$\rho_s \colon \Gamma(p, q, r) \to \mathrm{SL}_3(\C)$ with complex parameter~$s$ defined by
\begin{align*}
\rho_s(a) &=
\begin{pmatrix}
1 & 0 & 0 \\
-2 s \cos \frac{\pi}{p} & -1 & 0 \\
-2 \cos \frac{\pi}{r} & 0 & -1
\end{pmatrix}, \\
\rho_s(b) &=
\begin{pmatrix}
-1 & -2 s^{-1} \cos \frac{\pi}{p} & 0 \\
0 & 1 & 0 \\
0 & -2 \cos \frac{\pi}{q} & -1
\end{pmatrix}, \\
\rho_s(c) &=
\begin{pmatrix}
-1 & 0 & -2 \cos \frac{\pi}{r} \\
0 & -1 & -2 \cos \frac{\pi}{q} \\
0 & 0 & 1
\end{pmatrix}
\end{align*}
(the representations above are minor modifications of the ones introduced in \cite{Goldman},
where $\cos \frac{\pi}{p}$, $\cos \frac{\pi}{q}$ and $\cos \frac{\pi}{r}$ are
replaced by
$\cos \frac{2 \pi}{p}$, $\cos \frac{2 \pi}{q}$ and $\cos \frac{2 \pi}{r}$
respectively in the matrices).
A simple computation enables us to obtain the equation
\[ \tr \rho_s(abac) = 8 (s + s^{-1}) \cos \frac{\pi}{p} \cos \frac{\pi}{q} \cos \frac{\pi}{r} + 16 \cos^2 \frac{\pi}{p} \cos^2 \frac{\pi}{r} + 4 \cos^2 \frac{\pi}{q} - 1, \]
which shows that the restrictions of $\tr \rho_s $ to $\Delta(p, q, r)$ define
a nontrivial  affine curve contained in $X_3(\Delta(p, q, r))$.
Since the natural homomorphism $\pi_1(M) \to \pi_1^\mathrm{orb} (S^2(p, q, r))$ is surjective,
the morphism $X_3(\pi_1^\mathrm{orb} (S^2(p, q, r))) \to X_3(M)$
induced on the character varieties is an embedding.
Therefore one readily sees that, by identifying $X_3(\Delta(p,q,r))$
with $X_3(\pi_1^\mathrm{orb}(S^2(p,q,r)))$, the character variety
$X_3(M)$ also contains a nontrivial curve.

\section{Questions} \label{sc:Q}

We conclude with a list of questions.
Let $M$ be a compact, connected, irreducible and orientable $3$-manifold.
It is known by Boyer and Zhang \cite{BZ}, Motegi \cite{Motegi}, and
Schanuel and Zhang \cite{SZ} that there exists an essential surface not
given by any ideal points of any affine curves in $X_2(M)$ for a certain
$3$-manifold $M$.
We may now propose the following important question:

\begin{q} \label{q:detection}
Does there exist an essential surface $($without branched points$)$ 
not given by any ideal points of any affine curves in $X_2(M)$ but given 
by an ideal point of an affine curve in $X_n(M)$ for $n \geq 3$?
\end{q}

Here we remark that an essential surface (without any branched points) is 
also an essential tribranched surface in our terminology. [Note: as we have mentioned at the end of Section~\ref{sc:introduction}, (a much stronger form of) Question~\ref{q:detection} has been already solved affirmatively in \cite{FKN2}.]

\medskip

The next question concerns strongly essential tribranched surfaces, which we have defined and discussed in Subsection \ref{ssc:SETBS}.
\begin{q}
Under the same assumption as Theorem~$\ref{thm:main}$, is a strongly
 essential tribranched surface in $M$ also given by an ideal point of an
 affine curve in $X_n(M)$?
\end{q}

Now recall that we have imposed a little too strong assumption on the boundary of the 3-manifold $M$ under consideration in the proof of Theorem~\ref{thm:main}; namely we have assumed there that the boundary of $M$ is not empty when $n$ is strictly greater than $3$.
\begin{q}
Does the same conclusion as Theorem~$\ref{thm:main}$ hold without the
 assumption that the boundary $\partial M$ is non-empty when $n$ is
 strictly greater than $3$?   
\end{q}

Next let $M$ be a small Seifert manifold whose base orbifold is 
$S^2(p,q,r)$ with $p,q,r\geq 3$ (recall the definitions from
Section~\ref{sc:A}). 
Let us consider a theta graph $\Theta$ in $S^2(p, q, r)$, 
which has 2 vertices and 3 edges connecting them, 
so that all the cone points are separated by $\Theta$; see Figure~\ref{fig:Theta}. 
Then it is straightforward to see that the preimage $\Sigma_0$ of
$\Theta$ in $M$ under the projection $M\to S^2(p,q,r)$ 
is an essential tribranched surface, which seems to be a ``simplest'' one contained in $M$.

\begin{figure}[ht]
\centering
\begin{tikzpicture}

\draw [very thick] plot [domain=0:360,smooth] ({sqrt(16/3)*cos(\x)+sqrt(2/3)*sin(\x)},{sqrt(6)*sin(\x)});

\coordinate (A) at ({sqrt(2)},{-sqrt(2)});
\coordinate (B) at (-{sqrt(2)},{sqrt(2)});
\fill (A) circle [radius=3pt];
\fill (B) circle [radius=3pt];
\draw [very thick] (A) -- (B);

\coordinate (P) at (1,{sqrt(2.5)}) node at (P) {$\otimes$};
\node at (P) [right] {$\dfrac{2\pi}{p}$};
\coordinate (Q) at (-1.3,-{sqrt(1.8)}) node at (Q) {$\otimes$};
\node at (Q) [right] {$\dfrac{2\pi}{q}$};
\coordinate (R) at (3,-1) node at (R) {$\otimes$};
\node at (R) [right] {$\dfrac{2\pi}{r}$};
\end{tikzpicture}
\caption{A theta graph $\Theta$ separating the three cone points of $S^2(p,q,r)$}
\label{fig:Theta}
\end{figure}

\begin{q}
Is an essential tribranched surface, which is  given by an ideal point of
 the nontrivial curve considered in Section~$\ref{sc:A}$, isotopic to $\Sigma_0$?
\end{q}

The final question is concerning the characterisation of the class of $3$-manifolds containing essential tribranched surfaces.

\begin{q} \label{q:characterisation}
Does every aspherical $3$-manifold contain any essential tribranched surfaces or any strongly essential tribranched surfaces?
\end{q}


\end{document}